\documentclass[11pt]{article}
\usepackage{amsmath,amsthm, amssymb, amsfonts, mathrsfs, enumitem}
\usepackage[mathcal]{euscript} 
\usepackage[sort&compress,numbers,comma,square]{natbib}
\usepackage[active]{srcltx}
\usepackage[dvips]{color}
\usepackage[colorlinks]{hyperref}

\long\def\comment#1{}

\def\Grp#1{\left(#1\right)}
\def\Cbr#1{\left\{#1\right\}}
\def\Sbr#1{\left[#1\right]}
\def\Flr#1{\left\lfloor#1\right\rfloor}

\def\Norm#1{\left\|#1\right\|}
\def\cf#1{\mathbf{1}\Cbr{#1}}
\def\snorm#1{\|{#1}\|}

\def\cum#1#2{{#1}_1+\cdots+{#1}_{#2}}           
\def\eno#1#2{{#1}_1, \ldots, {#1}_{#2}}         

\def\gv{\,|\,}

\def\toi{\to\infty}

\def\dd{\mathrm{d}}

\def\Sp#1{\sp{(#1)}}

\def\nth#1{\frac{1}{#1}}

\def\Coms{\mathbb{C}}
\def\Ints{\mathbb{Z}}
\def\Nats{\mathbb{N}}
\def\Reals{\mathbb{R}}

\newtheorem{theorem}{Theorem}
\newtheorem{prop}{Proposition}
\newtheorem{cor}{Corollary}
\newtheorem{lemma}{Lemma}
\newtheorem{example}{Example}

\def\mean{\mathbb{E}}
\def\var{\mathbb{V}}
\def\pr{\mathbb{P}}
\def\ld{\lambda}  
\def\cpg{\gamma}  
\def\chf{\psi}    
\def\che{\Psi}    
\def\levy{\text{L\'evy}}
\def\rx{\epsilon}
\def\ess{\mathop{\rm ess}}
\def\tr{\mathrm{tr}}

\def\dgamma{\mathrm{Gamma}}
\def\dpois{\mathrm{Poisson}}
\def\dtv{d_{\rm TV}}
\def\dks{d_{\rm KS}}
\def\sppt{\text{sppt}}  

\def\rdf{\mathscr{S}} 
\def\Re{\mathrm{Re}}

\def\Pd{\partial}
\def\th{{\rm th}}
\def\tr{{\rm tr}}

\def\ip#1#2{\langle{#1},{#2}\rangle}    
\def\iunit{\mathrm{i}}                  
\def\ft{\widehat}                       
\def\std#1{{#1}^*}                      

\def\SRM{A}

\begin{document}
\begin{center}
  \large
  \textbf{Nonnormal small jump approximation of infinitely
    divisible distributions
  }
  \\[1.5ex]
  \normalsize
  Zhiyi Chi\\
  Department of Statistics\\
  University of Connecticut \\
  Storrs, CT 06269, USA \\[.5ex]
  E-mail: zhiyi.chi@uconn.edu \\[1ex]
  \today \footnote{Submitted for journal publication on March 28, 2013}
\end{center}

\begin{abstract}
  We consider a type of nonnormal approximation of infinitely
  divisible distributions that incorporates compound Poisson, Gamma,
  and normal distributions.  The approximation relies on achieving
  higher orders of cumulant matching, to obtain higher rates of
  approximation error decay.  The parameters of the approximation are
  easy to fix.  The computational complexity of random sampling of the
  approximating distribution in many cases is of the same order as
  normal approximation.  Error bounds in terms of total variance
  distance are derived.  Both the univariate and the multivariate
  cases of the approximation are considered.
  
  \medbreak\noindent

  \medbreak\noindent

\end{abstract}  

\section{Introduction}
Simulation of infinitely divisible (i.d.) random variables has many
applications.  In most cases, since closed formulas of i.d.\
distributions are unavailable, good approximation methods are desired.
Normal approximation of i.d.\ distributions, which was studied in
\cite{lorz:91:stat} and later developed in \cite{asmussen:01,
  cohen:07} in the framework of small jump approximation, has received
much attention in the literature \cite{veillette:12:spa,
  fournier:11:esps, kawai:11:jcam, dereich:11:aap, kohatsu:10:spa,
  hilber:09, baeumer:12:jap}.

The idea of small jump normal approximation is as follows.  Denote by
$\ld$ the $\levy$ measure of an i.d.\ random variable $X$.  Given
$r>0$, decompose $\ld = \ld_r + (\ld -\ld_r)$, such that $\ld-\ld_r\ge
0$ has finite mass.  Correspondingly, $X=X_r + \Delta_r$, where $X_r$
and $\Delta_r$ are independent i.d.\ random variables with $\levy$
measures $\ld_r$ and $\ld - \ld_r$, respectively.  Then $X_r$ is
approximated by a Gaussian random variable, while $\Delta_r$ is
sampled using standard methods for compound Poisson random variables.
Presumably, in order for the approximation to have a certain degree of
precision, the support of $\ld_r$ should be in a small neighborhood of
0.  The size of the neighborhood is controlled by $r$.  For the
univariate case, it is natural to set $r$ equal to the maximum jump
size \cite{asmussen:01}.  However, for the multivariate case, such use
of $r$ can be restrictive.  Generally speaking, one can use $r$ to
index any tunable quantity, as long as it controls (indirectly) the
size of the support of $\ld_r$ \cite{cohen:07}.

Normal approximation relies on second-order moment matching
between $X_r$ and a Gaussian random variable, by which we mean the
matching of their first and second moments.  This is equivalent to
second-order cumulant matching.  Without specifying details, by
certain measures, the error of the approximation in the univariate
case is bounded by 
\begin{align*}
  C |\kappa|_{3,X_r}/\kappa_{2,X_r}^{3/2},
\end{align*}
where $C$ is a universal constant, $\kappa_{2,X_r}$ is the second
cumulant of $X_r$, and for $j\ge 3$, $|\kappa|_{j,X_r} = \int |x|^j
\ld_r(\dd x)$ is the $j\th$ ``absolute cumulant'' of $X_r$
\cite{asmussen:01}.  The best currently available value of $C$ is
0.4785 (\cite{nourdin:peccati:12}, p.~71).   Also, in some symmetric
cases, since the third cumulants of $X_r$ and the Gaussian random
variable are 0, $|\kappa|_{3,X_r}$ in the bound can be more or less
replaced with $|\kappa|_{4,X_r}$, while the power of $\kappa_{2,
  X_r}$ is raised to 2 \cite{asmussen:01}.  From the pattern of the
bound, one could guess that, if $X_r$ and some $Y_r$ have the same
cumulants of order 1, \ldots, $q-1$ with $q\ge 5$, then $X_r$ could be
approximated by $Y_r$ with the error being bounded by
\begin{align*}
  C(r) (|\kappa|_{q, X_r} + |\kappa|_{q,Y_r})/\kappa_{2,X_r}^{q/2},
\end{align*}
where $C(r)$ is a near-constant, at least when $r$ is small, or most
ideally, a universal constant which may depend on the dimension of $X$
but is not very large.  Elementary calculations indicates that in many
cases, the above bound vanishes at a genuinely higher rate than the
bound for normal approximation as $r\to 0+$.  Since the $q\th$
cumulant of a Gaussian random variable is 0, the above bound, if true,
is consistent with the bound for normal approximation.

Even by some rough analysis on characteristic functions, there is good
reason to expect that the bound is true for both univariate and
multivariate cases.  However, before attempting to work out the
detail, perhaps one should first ask if such an approximation can
possibly be implemented easily.  The meaning of the question is
twofold.  First, the distribution of the approximating random variable
should be easy to identify; preferably, it is i.d.  Second, the
approximating random variable should be easy to sample; preferably,
the computational complexity of the sampling is of the same order as
the normal approximation.  If the answer to the question is positive,
then the next question is how large $q$ can be.  It can be anticipated
that the larger $q$ is, the faster the error of approximation vanishes
as $r\to 0+$.  After the answers to the two questions are in place, a
wide range of available techniques can potentially be modified to
establish the error bound (e.g.\ \cite{bhattacharya:76,
  barbour:chen:05, chen:goldstein:shao:11, nourdin:peccati:12}).

Clearly, cumulant matching is equivalent to moment matching.
Actually, our proof of the above type of bound eventually will be
based on moment matching.  However, thanks to the $\levy$-Khintchine
representation, it is more natural and convenient to consider 
cumulants than moments.  We shall show that it is fairly easy to
construct approximating i.d.\ random variables with matching cumulants
up to at least the fourth order, in other words, we can get
at least $q=5$.  In many important cases, we can get $q=6$, and in the
symmetric cases, we can get $q=10$.  For the univariate case, the
construction is particularly simple.  The approximating i.d.\ random
variable is the sum of a compound Poisson random variable and an
independent Gaussian random variable, with the former in turn being
the sum of a Poisson number of i.i.d.\ Gamma random variables.
Importantly, using algorithms already available \cite{devroye:86b,
  hormann:04}, the computational complexity of the random sampling for
the approximation is universally bounded, so it is of the same order
as the random sampling from a normal distribution.

We shall refer to the approximation as Poisson-Gamma-Normal (PGN)
approximation, although a longer name like ``compound Poisson-Normal
approximation with Gamma summands and higher order of cumulant
matching'' might be more appropriate.  We shall bound its error in
terms of total variation distance by combining Fourier analysis,
Lindeberg method (cf.\ \cite{chatterjee:06:ap} for a modern
application of it), and a device in \cite{asmussen:01}.  The results
are nonasymptotic and of the aforementioned type.  Asymptotically,
when applied to $X_r$, the approximation yields substantially higher
rate of precision than normal approximation as $r\to 0+$.  Of course,
on modern treatments of Poisson, compound Poisson, and normal
approximations, there is now an extensive literature, and on Gamma and
other types of approximations, there is also a growing literature; see
\cite{barbour:chen:05, chatterjee:11:aap, gotze:91:ap,
  chen:goldstein:shao:11, nourdin:peccati:12, reinert:05,
  nourdin:09:ptrf} and references therein.  However, it appears that
there has been little work on using convolutions of different types of
simple distributions to improve approximation, in the sense that the
error of approximation vanishes at a faster rate asymptotically.

For the multivariate case, the issue of approximation becomes quite
more involved, which is a well documented phenomenon
\cite{sazonov:68:sank, bhattacharya:76, gotze:91:ap, cohen:07,
  chen:goldstein:shao:11, nourdin:peccati:12}.  For normal
approximation of i.d.\ distributions, several important issues unique
to the multivariate case are identified and addressed in
\cite{cohen:07}.  The same issues also arise in the type of
approximation considered here and actually become more serious.  To
address them, we consider a ``radial'' cumulant matching approach.
Its idea is to apply the same cumulant matching method for the
univariate case to each radial direction in the $\levy$-Khintchine
representation, in such a way that, when the approximating $\levy$
measures and Gaussian measures along different radial directions are
``bundled'' together, we get a valid multivariate i.d.\ distribution
with desired order of cumulant matching and with the covariance of its
Gaussian component being precisely evaluated.  Although the approach
does not completely resolve the aforementioned issues, it seems to
work well in many important cases.

As in the univariate case, we shall prove a similar type of bound for
the error of the proposed PGN approximation in terms of total
variation distance.  On the other hand, the issue of computational
complexity needs to be considered more carefully.  Recall the 
approximation is applied to $X_r$ in the decomposition $X = X_r +
\Delta_r$.  As in the univariate case, the approximating random
variable for $X_r$ has a compound Poisson component.  Unfortunately,
this component now can only be sampled by summing a large number of
Poisson events.  As a result, the computational complexity of the
approximation of $X_r$ is much greater than normal approximation.
However, one has to take into account the computational complexity of
the sampling of $\Delta_r$.  We will argue using an 
example that although the proposed PGN approximation as a whole has
greater computational complexity than normal approximation,
asymptotically, as $r\to 0+$,  the two have the same order of
complexity.  Because the PGN approximation can yield substantially
higher rate of convergence, therefore, at least asymptotically, it is
worth the extra computation.  Note that whereas in \cite{cohen:07},
the focus is the approximation of the related $\levy$ processes, our
discussion is restricted to i.d.\ distributions.  An extension of PGN
approximation to processes will be subject to future work.

In Section \ref{sec:prelim}, we shall set up notation and collect
useful facts about i.d.\ distributions.  Sections \ref{sec:univariate}
and \ref{sec:multivariate} consider PGN approximation for univariate
i.d.\ distributions and multivariate i.d.\ distributions,
respectively.  The proofs of the main results in these two sections
are collected in Section \ref{sec:proof-main} and the proofs of
related technical results are collected in Section
\ref{sec:proof-aux}.

\section{Preliminaries} \label{sec:prelim}
\subsection{Notation}
Denote $\Ints_+=\{0\}\cup\Nats$ and $\Reals_+=[0,\infty)$.  If $f(x)$
is a function  on $\Reals^d$, where $x=(\eno x d)$, then by
$f\Sp\alpha(x)$ or $\Pd^\alpha f(x)$ we mean 1) $\alpha=(\eno\alpha
d)\in \Ints_+^d$, and 2) $f\Sp\alpha(x) = \Pd^{\alpha_1}_1 \cdots
\Pd^{\alpha_d}_d f(x)$, where $\Pd^k_j$ denotes the $k\th$-order
partial derivative with respect to $x_j$.  The order of $\alpha$ is
defined to be $|\alpha|=\cum \alpha d$.  Denote $x^\alpha =
x_1^{\alpha_1} \cdots x_d^{\alpha_d}$ and $\alpha! = \alpha_1! \cdots
\alpha_k!$.  Denote by $\rdf(\Reals^d)$ the space of rapidly
decreasing function on $\Reals^d$.  It is a basic fact that the
Fourier transform $h\to \ft h(t) = \int e^{\iunit \ip   t x} h(x)\,\dd
x$ is an homeomorphism of $\rdf(\Reals^d)$ onto itself
(\cite{grafakos:08}, p.~103).

Denote by $\sppt(\nu)$ the support of a measure $\nu$.  For two random
variables $X$ and $Y$, their total variation distance
\cite{barbour:chen:05} is denoted by
\begin{align*}
  \dtv(X,Y) = \sup\{\pr\{X\in A\} - \pr\{Y\in A\}: A \text{
    measurable}\}
\end{align*}
and, if $X$, $Y\in \Reals$, their Kolmogorov-Smirnov distance is
denoted by 
\begin{align*}
  \dks(X,Y) = \sup\{|\pr\{X\le x\} - \pr\{Y\le x\}|: x\in\Reals\}.
\end{align*}

For any i.d.\ random variable $X\in \Reals^d$, denote by $\chf_X$ and
$\che_X$ its characteristic function and characteristic exponent,
respectively
\begin{align*}
  \chf_X(t) = e^{-\che_X(t)} = \mean [e^{\iunit\ip t X}],
  \quad
  t\in\Reals^d.
\end{align*}
Let $f_X$ be the probability density of $X$.  If it exists, then
$\chf_X = \ft f_X$.  Denote
\begin{align*}
  \kappa_{\alpha, X} = \left.\Pd^\alpha \ln\mean[e^{\ip t
      X}]\right|_{t=0}, \quad
  |\kappa|_{\alpha, X} = \int |u^\alpha|\,\ld(\dd u).
\end{align*}
The quantity $\kappa_{\alpha, X}$ is known as the $\alpha\th$ cumulant
of $X$.  It is well defined provided that $\mean [e^{\ip t X}]<\infty$
for all $t$ in a neighborhood of 0.  Some properties of cumulants can
be found in \cite{nourdin:peccati:12}.  We shall refer to
$|\kappa|_{\alpha, X}$ as the $\alpha\th$ absolute cumulant of $X$. 

\subsection{Basic assumptions and facts} \label{ss:assumptions}
Let $X\in \Reals^d$ be i.d.\ with $\levy$ measure $\ld$.  We will
always assume
\begin{align}  \label{e:ID}
  \che_X(t) = \int (1+\iunit \ip t u
  -e^{\iunit\ip t u})\, \ld(\dd u), \quad t\in \Reals^d,
\end{align}
in particular, $X$ has no Gaussian component and $\mean X=0$.  The
assumption causes no loss of generality since, if necessary, we can
decompose $X$ as $X'+X''$, such that $X'$ has a  $\levy$ measure
satisfying \eqref{e:ID} and $X''$ is a compound Poisson random
variable.  Then we can take $X' - \mean X'$ as the new $X$.  We will
also always assume
\begin{align*}
  \ld(\Reals^d)=\infty.
\end{align*}
Under the assumption, $X$ is not compound Poisson and $\pr\{X=x\}=0$
for $x\in\Reals^d$ (\cite{sato:99}, Theorem 27.4).  It is known
that if $d=1$ and $\ld(\Reals) < \infty$ then $X$ does not admit
normal approximation \cite{asmussen:01}.  The assumption excludes the
case of lattice valued i.d.\ random variables, for which
Poisson-Charlier approximation has been studied \cite{lorz:91:stat,
  barbour:chen:05}.

Recall that for any $a>0$, $\mean \snorm{X}^a < \infty$ if and only if
$\int \cf{\snorm u>1}\snorm u^a\,\ld(\dd u)<\infty$ (\cite{sato:99},
p.~159--160).  If $\ld$ is a Borel measure on $\Reals^d$ with
$\ld(\{0\})=0$ and $\int (\snorm u^2\wedge 1)\, \ld(\dd u)<\infty$,
then $\ld$ is the $\levy$ measure of some i.d.\ random variable
(\cite{sato:99}, Theorem 8.1).  Also, if $\sppt(\ld)$ is bounded, then
$\mean[e^{\ip t X}]<\infty$ for all $t\in\Reals^d$ (\cite{sato:99},
Theorem 25.17).  By differentiation, 
\begin{align*}
  \kappa_{\alpha, X} = 
  \begin{cases}
    0 & |\alpha|=1\\
    \int u^\alpha\,\ld(\dd u) & |\alpha|>1
  \end{cases}
\end{align*}
for all $\alpha$.  It is easy to see that if each $\alpha_i$ is even
or $\sppt(\ld)\in \Reals_+^d$, then $\kappa_{\alpha, X} =
|\kappa|_{\alpha,X}$.  Also,
\begin{align*}
  \var(X) = \int uu'\,\ld(\dd u), 
\end{align*}
and hence $\tr(\var(X)) = \int \snorm u^2\,\ld(\dd u)$, where $\tr(A)$
denotes the trace of a square matrix $A$.

Let $f\in C(\Reals^d)$.  For $n\in \Nats$, let $U_n$ be i.d.\ with
$\che_{U_n} = n^{-1} \che_X$.  If $\sppt(f) 
\subset \Reals^d \setminus \{0\}$ and is compact, then $n\mean f(U_n)
\to \int f\,\dd\ld$, which directly follows from the vague convergence
of $n\pr\{U_n\in\dd x\}$ to $\ld(\dd x)$ on $\{x: \snorm x>\rx\}$
given $\rx>0$ (\cite{bertoin:96}, p.~39); see \cite{kallenberg:02}
for detail on vague convergence.  The next result, which will be used
later, concerns the case where $\sppt(f)$ is not a compact set in
$\Reals^d\setminus\{0\}$.  When $d=1$ and $f(x) = |x|^p$ with $p>2$,
the result is established as Lemma 3.1 in \cite{asmussen:01}.  
However, as seen from the case $X\sim N(0,1)$, the asserted
convergence in general is not true if $f(x) = x^2$.
\begin{prop} \label{prop:AR}
  Suppose $|f(x)|\le g(\snorm x)$, where $g\in C(\Reals_+)$ is
  nondecreasing with $g(t) = o(t^2)$ as $t\to 0+$.  Suppose $\mean
  \snorm X^2<\infty$ and one of the following holds, 1) $\mean
  g(c\snorm{X'})<\infty$ for some $c>1$, where $X'=X_1-X_2$, with
  $X_i$ i.i.d.\ $\sim X$, 2) $\mean g(2 c\snorm{X})<\infty$ for some
  $c>1$, or 3) provided $X$ is symmetric, $\mean g(\snorm{X})<\infty$.
  Then $f \in L^1(\ld)$ and $n \mean f(U_n)\to \int f\,\dd\ld$ as
  $n\toi$.
\end{prop}

\section{Univariate Poisson-Gamma-Normal approximation} 
\label{sec:univariate}
\subsection{Cumulant matching}
For simplicity and without loss of generality, we will only consider
two cases 1) $\sppt(\ld)\subset \Reals_+$ and 2) $X$ is symmetric.
First, suppose $\sppt(\ld)\subset \Reals_+$.  Given $r>0$,
decompose $X=X_r + \Delta_r$, where $X_r$ and $\Delta_r$ are
independent i.d.\ random variables such that
\begin{align}  \label{e:ID-r}
  \che_{X_r}(t)
  =
  \int (1+\iunit t u-e^{\iunit t u})\, \ld_r(\dd u), \quad
  \text{with}\ \ld_r(\dd u) = \cf{u<r} \ld(\dd u).
\end{align}
Given $p\ge -1$, let $Y_r$ be an i.d.\ random variable with
\begin{align} \label{e:Gamma-r}
  \che_{Y_r}(t) =
  \int_0^\infty (1+\iunit t u-e^{\iunit t u} ) \cpg_r(\dd u),
  \quad\text{with}\
  \cpg_r(\dd u) = m(r) u^p e^{-u/s(r)} \,\dd u,
\end{align}
where $m(r)>0$ and $s(r)>0$ are constants that need to be determined.
Finally, let
\begin{align} \label{e:pgn-r}
  T_r = Y_r + \sigma(r) Z, \quad Z\sim N(0,1) \text{ independent of }
  Y_r,
\end{align}
where $\sigma(r)>0$ is a constant that needs to be determined.

We shall use $T_r+\Delta_r$ to approximate $X$, or equivalently, use
$T_r$ to approximate $X_r$.  But first, let us point out how easy
it is to sample $T_r$.  Clearly, the issue is the sampling of $Y_r$.
Since $Y_r= U-\mean U$, where $U\ge 0$ is i.d.\ with $\levy$
density $m(r)\cf{u>0} u^p e^{-u/s(r)}$, and $\mean U = \Gamma(p+2)
m(r) s(r)^{p+2}$, we only need to consider the computational
complexity of the sampling of $U$.  If $p=-1$, then $U\sim
\dgamma(m(r), s(r))$, the Gamma distribution with shape parameter
$m(r)$ and scale parameter $s(r)$.  It is known that the sampling of
$\dgamma(a,b)$ has universally bounded complexity regardless of
$(a,b)$ (\cite{devroye:86b}, p.~407--420).  If $p>-1$, then
$U\sim\sum_{i=1}^ N\xi_i$, where $N\sim \dpois(a)$ with
$a=\int_0^\infty m(r) u^p e^{-u/s(r)}\,\dd u = \Gamma(p+1) m(r) 
s(r)^{p+1}$, and $\xi_i$ are i.i.d.\ $\dgamma(p+1, s(r))$ random
variables independent of $N$.  The sampling of $\dpois(a)$ is known to
have universally bounded complexity (\cite{devroye:87} or
\cite{hormann:04}, p.~228--241).  On the other hand, conditional on
$N$, $U\sim \dgamma(N(p+1), s(r))$.  Therefore, the sampling of $U$,
and hence that of $T_r$, has the same order of complexity as the
sampling of a normal random variable.

Due to the $\levy$-Khintchine representation of $T_r$, we refer to
the approximation of $X$ by $T_r+\Delta_r$, or $X_r$ by $T_r$, as
Poisson-Gamma-Normal (PGN) approximation.

It is easy to see $\mean X_r = \mean T_r = \mean Y_r=0$, and for $j\ge
2$,
\begin{align}
  \begin{array}{c}
    \kappa_{j,X_r} = \int u^j \, \ld_r(\dd u), \quad
    \kappa_{j,T_r} = \kappa_{j,Y_r} + \cf{j=2}\sigma(r)^2,
    \\[1ex]
    \text{with}\ \ 
    \kappa_{j,Y_r} =\Gamma(j+p+1) m(r) s(r)^{j+p+1}.
  \end{array}
  \label{e:cum-X-Y}
\end{align}
This is the starting point of cumulant matching between $X_r$ and
$T_r$.  In the next result, we allow $r=\infty$, so it applies to any
i.d.\ random variable with finite fourth cumulant.

\begin{prop}[Fourth-order cumulant matching] \label{prop:cum-match}
  Fix $0<r\le \infty$.  If $r=\infty$, also assume $\kappa_{4,
    X_r}<\infty$.  Then for all large $p$,
  \begin{align} \label{e:pgn-p}
    \frac{p+4}{p+3} <
    \frac{\kappa_{2,X_r}\kappa_{4,X_r}}{\kappa_{3,X_r}^2}.
  \end{align}
  For any $p\ge -1$ satisfying \eqref{e:pgn-p}, if
  \begin{align}
    s(r) = \frac{\kappa_{4, X_r}}{(p+4)\kappa_{3,X_r}},
    \quad
    m(r) = \frac{\kappa_{3, X_r}}{\Gamma(p+4) s(r)^{p+4}},
    \label{e:pgn-s}
  \end{align}
  and if $Y_r$ is defined by \eqref{e:Gamma-r}, then $\kappa_{2,X_r} >
  \kappa_{2,Y_r}$, and by setting
  \begin{align}
    \sigma(r) = (\kappa_{2,X_r}- \kappa_{2,Y_r})^{1/2},
    \label{e:pgn-v}
  \end{align}
  $\kappa_{j,X_r} = \kappa_{j,T_r}$ for $2\le j\le 4$.
\end{prop}
\begin{proof}
  By assumption, $\kappa_{i,X_r}<\infty$ for $2\le i\le 4$.  By
  H\"older inequality, $\kappa_{3,X_r}^2<\kappa_{2,X_r} \kappa_{4,
    X_r}$, which implies \eqref{e:pgn-p}.  From \eqref{e:cum-X-Y}, by
  setting $s(r)$ and $m(r)$ as in \eqref{e:pgn-s}, $\kappa_{j,X_r} =
  \kappa_{j, Y_r}$ for $j=3, 4$ and 
  \begin{align*}
    \kappa_{2,Y_r} = \Gamma(p+3) m(r) s(r)^{p+3} = \Gamma(p+3)
    \frac{\kappa_{3,X_r}}{\Gamma(p+4) s(r)}
    = \frac{(p+4)\kappa_{3,X_r}^2}{(p+3) \kappa_{4,X_r}}.
  \end{align*}
  Then for $p\ge -1$  satisfying \eqref{e:pgn-p}, $\kappa_{2,Y_r}
  < \kappa_{2,X_r}$.  The rest of the result is then clear.
\end{proof}

\begin{prop}[Fifth-order cumulant matching] \label{prop:cum-match2}
  Let $\ld(\dd u) = \cf{u>0}u^{-a-1} \ell(u)\,\dd u$, where $a\in
  (0,2)$ and $\ell(u)$ is slowly varying at $0+$.  Let $p = p(r)$ be
  defined by the equation
  \begin{align*}
    1+\nth{p+4} =
    \frac{\kappa_{3, X_r} \kappa_{5, X_r}}{\kappa_{4,X_r}^2}.
  \end{align*}
  Then for all small $r>0$, $p>-1$ and satisfies \eqref{e:pgn-p}, and
  by setting $s(r)$, $m(r)$ and $\sigma(r)$ according to
  \eqref{e:pgn-s} and \eqref{e:pgn-v}, $\kappa_{j,X_r} = \kappa_{j,
    T_r}$ for $2\le j\le 5$. 
\end{prop}
\begin{proof}
  Since $\ell$ is slowly varying at $0+$, for $j\ge 3$,
  \begin{align}  \label{e:svf-power}
    \kappa_{j,X_r} = \int_0^r u^{j-1-a}
    \ell(u)\,\dd u \sim \frac{r^{j-a} \ell(r)}{j-a},
  \end{align}
  as $r\to 0+$ \cite{bertoin:96, korevaar:04}.  As a result, 
  \begin{align*}
    \nth{p+4}
    = \frac{\kappa_{3,X_r}\kappa_{5,X_r}}{\kappa_{4,X_r}^2}-1
    \sim \frac{(4-a)^2}{(3-a)(5-a)} -1 = \nth{(3-a)(5-a)},
    \quad r\to 0+.
  \end{align*}
  It follows that $p\sim a^2 - 8a+11 >-1$.  Thus, for all small
  $r>0$, $p>-1$.  By Proposition \ref{prop:cum-match}, it only remains
  to show that $p$ satisfies \eqref{e:pgn-p} and $\kappa_{5,X_r} =
  \kappa_{5, Y_r}$.  By \eqref{e:svf-power}, as $r\to 0+$,
  \begin{align*}
    \frac{\kappa_{2,X_r}\kappa_{4,X_r}}{\kappa_{3,X_r}^2}
    \sim \frac{(3-a)^2}{(2-a)(4-a)} = 1+\nth{a^2 - 6 a + 8}.
  \end{align*}
  Therefore, with $p>-1$, \eqref{e:pgn-p} is equivalent to $p>a^2 -
  6a+5$, which holds for $a\in (0,2)$.  Finally, that $\kappa_{5, X_r}
  = \kappa_{5, Y_r}$ follows from $\kappa_{3,X_r} \kappa_{5,X_r}
  /\kappa_{4,X_r}^2 = (p+5)/(p+4) = \kappa_{3,Y_r} \kappa_{5,Y_r}/
  \kappa_{4,Y_r}^2$.
\end{proof}

Now we consider the symmetric case.  Suppose $X = X\Sp 1 - X\Sp 2$,
where $X\Sp i$ are i.i.d.\ with $\levy$ measure $\ld$ supported in
$\Reals_+$.  Let $X_r = X\Sp 1_r - X\Sp 2_r$, and approximate it 
by $T_r = T_r\Sp 1 - T_r\Sp 2$, where $T_r\Sp i$ are i.i.d.\ defined
in \eqref{e:pgn-r}.  Since all the odd-ordered cumulants of 
$X_r$ and $T_r$ are 0, we only need to match their even-ordered
cumulants.  The next results states that for the general case, we can
match their cumulants up to order 7, and for the i.d.\ distribution as
in Proposition \ref{prop:cum-match2}, we can match their cumulants up
to order 9.

\begin{prop}[Symmetric case] \label{prop:cum-match-symm}
  1) Fix $r>0$.  Then for all large $p$,
  \begin{align} \label{e:pgn-p-symm}
    \frac{(p+5)(p+6)}{(p+3)(p+4)} < 
    \frac{\kappa_{2,X_r} \kappa_{6,X_r}}{\kappa_{4,X_r}^2}.
  \end{align}
  For any $p\ge -1$ satisfying \eqref{e:pgn-p-symm}, if
  \begin{align} \label{e:pgn-s-symm}
    s(r) = \sqrt{
      \frac{\kappa_{6, X_r}}{(p+5)(p+6)\kappa_{4, X_r}}
    }, \quad
    m(r) = \frac{\kappa_{4, X_r}}{2\Gamma(p+5) s(r)^{p+5}},
  \end{align}
  and if $Y_r = Y\Sp 1_r - Y\Sp 2_r$, where $Y\Sp i_r$ are i.i.d.\
  as defined in \eqref{e:Gamma-r}, then $\kappa_{2,X_r} >
  \kappa_{2,Y_r}$, and by setting  $\sigma(r)$ as in \eqref{e:pgn-v},
  $\kappa_{j,X_r} = \kappa_{j,T_r}$ for $2\le j\le 7$.

  2) If the $\levy$ measure $\ld$ of $X\Sp i$ is $\cf{u>0} u^{-a-1}
  \ell(u)\,\dd u$, where $a\in (0,2)$ and $\ell(u)$ is slowly varying
  at $0+$, then for all small $r>0$, there is a unique $p=p(r)>0$
  satisfying \eqref{e:pgn-p-symm} and
  \begin{align} \label{e:pgn-p-symm2}
    \frac{(p+7)(p+8)}{(p+5)(p+6)} = 
    \frac{\kappa_{4, X_r} \kappa_{8, X_r}}{\kappa_{6,X_r}^2}.
  \end{align}
  Consequently, for this $p$, by setting $s(r)$ and $m(r)$ according
  to \eqref{e:pgn-s-symm} and $\sigma(r)$ according to
  \eqref{e:pgn-v}, $\kappa_{j,X_r} = \kappa_{j,T_r}$ for $2\le j\le
  9$.
\end{prop}
\begin{proof}
  1)  By H\"older inequality, $\kappa_{4, X_r}^2 < \kappa_{2, X_r}
  \kappa_{6, X_r}$, so for all large $p$, \eqref{e:pgn-p-symm} is
  satisfied.  Since for even-valued $j$, $\kappa_{j,Y_r} = 2
  \kappa_{j, Y_r\Sp 1} = 2 \Gamma(j+p+1) m(r)s(r)^{j+p+1}$, it is easy
  to see $\kappa_{4,X_r} = \kappa_{4,Y_r}$ and $\kappa_{6, X_r} =
  \kappa_{6, Y_r}$.  On the other hand, for all odd-valued $j$,
  $\kappa_{j, X_r} = \kappa_{j, Y_r}=0$.   Finally, by similar
  argument for Proposition \ref{prop:cum-match}, $\kappa_{2,Y_r} <
  \kappa_{2, X_r}$, leading to $\kappa_{j, X_r} = \kappa_{j, T_r}$ for
  $2\le j\le 7$.

  2) Following the proof of Proposition \ref{prop:cum-match2},
  \begin{align*}
    \frac{\kappa_{4, X_r} \kappa_{8, X_r}}{\kappa_{6,X_r}^2}
    \sim \frac{(6-a)^2}{(4-a)(8-a)} = 1+\frac{4}{(4-a)(8-a)}:= h(a),
  \quad r\to 0+.
  \end{align*}
  Clearly, $h(a)$ is strictly increasing on $(0,2)$.
  On the other hand,
  \begin{align*}
    g(p):=\frac{(p+7)(p+8)}{(p+5)(p+6)} = \Grp{1+\frac{2}{p+5}}
    \Grp{1+\frac{2}{p+6}}
  \end{align*}
  is strictly decreasing on $(-1,\infty)$, with $g(0) > h(2) > h(a) >
  h(0) > 1=g(\infty)$.  Therefore, there is a unique $p>0$ satisfying
  \eqref{e:pgn-p-symm2}.  We have to show that for this $p = p(r)$,
  \eqref{e:pgn-p-symm} is satisfied for all small $r>0$.  By
  continuity, it suffices to show that for $p>0$,
  \begin{align*}
    \frac{(p+7)(p+8)}{(p+5)(p+6)} = \frac{(6-a)^2}{(4-a)(8-a)}
    \implies
    \frac{(p+5)(p+6)}{(p+3)(p+4)} < \frac{(4-a)^2}{(2-a)(6-a)}.
  \end{align*}
  By calculation, the equality is equivalent to $2p^2 = 2p(a^2
  - 12 a + 21) + 13 a^2 - 156 a + 356$, while the inequality is
  equivalent to $2p^2 > 2p(a^2 - 8 a + 5) + 9a^2 - 72 a + 84$.  Then,
  by $p>0$ and $0<a<2$, the equality indeed implies the inequality.
  The rest of the proof then follows the one for 1). 
\end{proof}

Propositions \ref{prop:cum-match2} and \ref{prop:cum-match-symm}
directly lead to the following result on the truncated stable case.
Note that for the non-truncated case, simple exact sampling method is
known \cite{devroye:86b}.  Also, for $a\in (0,1)$, the truncated
stable distribution can be sampled exactly \cite{chi:12}.  

\begin{cor} \label{cor:cum-match2}
  Let $\ld(\dd u) = c\cf{0<u<r_0}u^{-a-1}\,\dd u$, where $c>0$,
  $r_0\in (0,\infty)$, and $a\in (0,2)$.

  1) Suppose $X$ has $\levy$ measure $\ld$.  If $p = a^2 - 8 a + 11$,
  then $p>-1$ and for all $0<r\le r_0$, by setting $s(r)$ and $m(r)$
  according to \eqref{e:pgn-s}, $\kappa_{2, X_r}>\kappa_{2,Y_r}$, and
  by setting $\sigma(r)$ according to \eqref{e:pgn-v}, $\kappa_{j,X_r}
  = \kappa_{j, T_r}$ for $2\le j\le 5$.

  2) Suppose $X=X\Sp 1 - X\Sp 2$, with $X\Sp i$ being i.i.d.\ with
  $\levy$ measure $\ld$.  If $p$ is the unique solution in
  $(0,\infty)$ to 
  \begin{align*}
    \frac{(p+7)(p+8)}{(p+5)(p+6)} = \frac{(6-a)^2}{(4-a)(8-a)},
  \end{align*}
  then for all $0<r\le r_0$, by setting $s(r)$, $m(r)$ according to
  \eqref{e:pgn-s-symm}, $\kappa_{2, X_r}>\kappa_{2,Y_r}$, and by
  setting $\sigma(r)$ according to \eqref{e:pgn-v}, $\kappa_{j,X_r} =
  \kappa_{j, T_r}$ for $2\le j\le 9$.
\end{cor}

\subsection{Error bound for approximation} \label{ss:error-1d}
We consider the error of approximation of $X$ by $\Delta_r + T_r$. 
Denote constants
\begin{align*}
  C_1
  &=
  \sin 1 = 0.841\ldots,
  \\
  C_2
  &=
  \inf_{p>0} \nth{\Gamma(p)} \int_0^p u^{p-1} e^{-u}\,\dd u
  = \inf_{p>0} \pr\{\xi_p \le p\}, \quad \xi_p\sim \dgamma(p,1). 
\end{align*}
Note that $C_2\in (0,1)$ because
\begin{align*}
  \nth{\Gamma(p)} \int_0^p u^{p-1} e^{-u}\,\dd u
  \sim  \nth{\Gamma(p)} \int_0^p u^{p-1} \,\dd u \sim p^p \to 1, \quad
  p\to 0+
\end{align*}
and by Central Limit Theorem, $\pr\{\xi_p - p \le 0\} \to 1/2$ as
$p\toi$.  

Observe that for $s(r)$ defined in \eqref{e:pgn-s} or
\eqref{e:pgn-s-symm}, $s(r)< r/(p+3)$.  The main result of the
section is the following.
\begin{theorem} \label{thm:tv}
  Fix $r\in (0,\infty)$.  Let $T_r$ be defined by \eqref{e:Gamma-r}
  -- \eqref{e:pgn-r} for the asymmetric case, or by $T_r\Sp 1 - T_r\Sp
  2$ for the symmetric case, with $T_r\Sp i$ i.i.d.\ defined by
  \eqref{e:Gamma-r} -- \eqref{e:pgn-r}.  Suppose $s(r) < r/(p+3)$
  and $\sigma(r)>0$.  Let
  \begin{align*}
    L(t,r)
    &=
    \frac{t^2}{2} \min\{C_1^2 \kappa_{2,X_{1/|t|}},
      \ \sigma(r)^2\}.
  \end{align*}
  For $j\ge 1$, define $Q_j(r)\ge 0$ such that
  \begin{align*}
    Q_j(r)^2 = \frac{\Gamma(j+1/2)}{2 (C_1^2 C_2)^{j+1/2}}
    + \kappa_{2,X_r}^{j+1/2}
    \int_{1/r}^\infty t^{2j} e^{-2 L(t,r)}\,\dd t.
  \end{align*}
  Given $q\ge 5$, suppose $\kappa_{j,X_r} = \kappa_{j, T_r}$
  for $2\le j<q$.  Then
  \begin{align}
    \dtv(X, \Delta_r+T_r)\le 
    \frac{|\kappa|_{q,X_r} + |\kappa|_{q,Y_r}}{q!\kappa_{2,X_r}^{q/2}}
    [q Q_{q-1}(r) + Q_q(r) + Q_{q+1}(r)].
    \label{e:pgn-dtv}
  \end{align}
\end{theorem}

\noindent{\it Remark.}
\begin{enumerate}[itemsep=0ex, topsep=.5ex] 
\item The bound is on $\dtv$ instead of the more commonly used $\dks$
  \cite{asmussen:01, lorz:91:stat}.  However, we have not been able to
  derive a Berry-Esseen type of bound of the form $C(|\kappa|_{q,X_r}
  + |\kappa|_{q,Y_r})/\kappa_{2,X_r}^{q/2}$, with $C$ a universal
  constant only depending on $q$.  It appears that some key
  ingredients for the proof of the Berry-Esseen bound for normal
  approximation are still missing for higher order approximations.
  Also, it is likely that the constants in the bounds are not optimal.

\item The bound will be proved by combining Fourier analysis, the
  Lindeberg method, and a device in \cite{asmussen:01} (cf.\ the
  proof of Theorem 25.18 in \cite{sato:99}).  Although a bound on
  $\dks$ may be established solely based on Fourier analysis
  \cite{chow:teicher:97, lorz:91:stat}, our proof seems to be more
  transparent and suitable for generalization to multivariate
  cases.
\end{enumerate}

In the bound for $\dtv(X, \Delta_r+T_r)$, $Q_j(r)$ look rather
technical.  We can use the following result to bound them.
\begin{prop} \label{prop:exp-L}
  For $b\in (0,1)$ and $q\ge 3$, there is $M=M(b,q)>0$, such that if
  \begin{align} \label{e:X-Y-moment}
    \limsup_{r\to 0+} \frac{\kappa_{2,Y_r}}{\kappa_{2,X_r}}<b,
    \quad
    \liminf_{r\to 0+} \frac{\kappa_{2,X_r}}{r^2\ln (1/r)}> M,
  \end{align}
  then for any $2\le j\le q+1$,
  \begin{align*}
    Q_j(r)^2 = \frac{\Gamma(j+1/2)}{2 (C_1^2 C_2)^{j+1/2}} + o(1),
    \quad
    r\to 0+.
  \end{align*}
\end{prop}

Since the proof is short, we give it here.  By \eqref{e:X-Y-moment},
for all small $r>0$, $\sigma(r)^2 = \kappa_{2,X_r} -\kappa_{2,Y_r} >
(1-b)\kappa_{2, X_r}$.  Then, from the increasing monotonicity of
$\kappa_{2,X_r}$ in $r$, there is a constant $c=c(b)>0$, such that for
$t\ge 1/r$, $L(t,r)\ge ct^2 \kappa_{2,X_{1/t}}$.  Consequently, if
$M\ge(q+2)/c$, then by \eqref{e:X-Y-moment}, for $t\ge 1/r$, $L(t,r)
\ge M c \ln t\ge (q+2)\ln t$, and hence for all $2\le j\le q+1$,
\begin{align*}
  \int_{1/r}^\infty t^{2j} e^{-2L(t,r)}\,\dd t
  \le \int_{1/r}^\infty t^{2(q+1)-2M c}\,\dd t = o(1),
  \quad r\to 0+.
\end{align*}
Since $\kappa_{2, X_r}=o(1)$ as $r\to 0+$, the proof is complete.

\begin{example} \rm \label{ex:stable}
  Let $\ld(\dd u) = c \cf{0<u<r_0} u^{-a-1}\,\dd u$, where $c>0$,
  $0<r_0<\infty$, and $a\in (0,2)$.  By Corollary
  \ref{cor:cum-match2}, given $r\in (0, r_0)$, if $p= a^2 - 8 a + 11$,
  and $s(r)$, $m(r)$ and $\sigma(r)$ are set according \eqref{e:pgn-s}
  -- \eqref{e:pgn-v}, then $\kappa_{j,X_r}=\kappa_{j,T_r}$ for $2\le j
  < q=6$.  To apply \eqref{e:pgn-dtv}, we need to get
  $\kappa_{2,X_r}$, $|\kappa|_{6,X_r}=\kappa_{6,X_r}$, and
  $|\kappa|_{6,Y_r}=\kappa_{6,Y_r}$.  For $j\ge 2$, $\kappa_{j,X_r} =
  c r^{j-a}/(j-a)$.  Since
  \begin{align*}
    s(r)
    &=
    \frac{\kappa_{4,X_r}}{(p+4)\kappa_{3,X_r}}=
    \frac{(3-a) r}{(p+4)(4-a)} = \frac{r}{(4-a)(5-a)},
    \\
    \kappa_{2,Y_r}
    &= \frac{\kappa_{3,Y_r}}{(p+3) s(r)}
    = \frac{\kappa_{3,X_r}}{(p+3) s(r)}
    = \frac{c(4-a)(5-a) r^{2-a}}{(3-a)(a^2 - 8a+14)},
  \end{align*}
  then $\kappa_{6,Y_r}=(6+p) s(r) \kappa_{5,Y_r} = (6+p)
  s(r)\kappa_{5,X_r}= c A(a) r^{6-a}$, with
  \begin{align*}
    A(a)
    =
    \frac{a^2 - 8a +17}{(4-a)(5-a)^2}.
  \end{align*}    
  Therefore, by Theorem \ref{thm:tv},
  \begin{align*}
    \dtv(X, \Delta_r+T_r) \le
    \frac{(2-a)^3}{c^2}
    \Sbr{\nth{6-a} + A(a)}\times
    \frac{6Q_5(r)+Q_6(r)+Q_7(r)}{6!} \times r^{2a}.
  \end{align*}

  Since $0<\kappa_{2,Y_r}/\kappa_{2,X_r}<1$ is a constant independent
  of $r$, and $\ld$ satisfies Orey's condition $\liminf_{r\to 0+}
  \kappa_{2,X_r}/ r^{2-a} >0$ (\cite{orey:68:ams}; also see
  \cite{sato:99}, Proposition 28.3),  the conditions in
  \eqref{e:X-Y-moment} are satisfied no matter the value of $M$.  Then
  by Proposition \ref{prop:exp-L}, $\dtv(X, \Delta_r+T_r) =
  O(r^{2a})$.  This may be compared to the normal approximation in
  \cite{lorz:91:stat, asmussen:01}, where $\dks$ between $X$ and its
  normal approximation is of rate $O(r^{a/2})$ when $X$ is
  asymmetric.

  Furthermore, if $X = X\Sp 1 - X\Sp 2$ is symmetric, where $X\Sp i$
  are i.i.d.\ with $\levy$ measure $\ld$, then by similar argument
  while using 2) of Corollary  \ref{cor:cum-match2}, it can be
  seen that we can set $q=10$ and get $\dtv(X, \Delta_r+T_r) =
  O(r^{4a})$, whereas the $\dks$ between $X$ and its normal
  approximation in this case is of rate $O(r^a)$ \cite{asmussen:01}.
  \qed
\end{example}

\begin{example} \label{ex:stable-tilted} \rm
  Let $\ld(\dd u) = \cf{u>0} u^{-a-1} \exp(-u^b) \,\dd u$, where $a\in
  (0,2)$ and $b>0$.  If we directly evaluate $\int_0^r u^j \ld(\dd u)$
  for $j\ge 2$, there is no closed formulas available.  The following
  method avoids the problem.  Recall that for any odd positive integer
  $n$, $e^{-u} \ge f_n(u)$ for $u\ge 0$, where $f_n(u) = \sum_{i=0}^n
  (-u)^i/i!$.  Let $n\ge 1$ be the smallest odd number greater than
  $a/b-1$ and $F(u) = \cf{0<u<r_0} f_n(u^b)$, where $r_0 = \sup\{r>0:
  f_n(u)>0$ for all $0\le u<r^b\}$.  Decompose $\ld(\dd u) = \ld_1(\dd
  u) + \ld_2(\dd u)$, where $\ld_1(\dd u) = \cf{u>0} u^{-a-1}
  F(u)\,\dd u$.  Because $u^{-a-1} [\exp(-u^b) - F(u)] =
  O(u^{(n+1)b-a-1})$ as $u\to 0+$, $\ld_2$ has finite mass, and hence
  corresponds to a compound Poisson random variable that can be
  sampled exactly.  Since $\ld_1(\dd u) = \cf{0<u<r_0} u^{-a-1}
  f_n(u^b)$, for $0<r<r_0$, it is easy to evaluate $\int_0^r u^j
  \ld_1(\dd u)$.  Then we can apply PGN approximation to $\ld_1$.  If
  $X$, $X'$, and $X''$ denote i.d.\ random variables with $\levy$
  measures $\ld$, $\ld_1$, and $\ld_2$, respectively, and $\Delta_r$
  and $T_r$ the i.d.\ random variables from the approximation, then by
  Proposition \ref{prop:cum-match2}, we can get $\dtv(X, \Delta_r +
  T_r + X'') \le \dtv(X', \Delta_r + T_r) = O(r^{2a})$.
\end{example}

\begin{example} \label{ex:ln} \rm
  Let $\ld(\dd u) = c \cf{0<u<1} u^{-1} \ln(1/u)\,\dd u$.  Since
  \begin{align*}
    \int_{u<r} u^2 \,\ld(\dd u) = c\int_0^r u\ln(1/u)\,\dd u
    = \frac{c r^2 [2\ln(1/r) + 1]}{4}
  \end{align*}
  by Proposition 2.1 in \cite{asmussen:01}, normal approximation
  works in the sense that its error in terms of $\dks$ tends to 0 as
  $r\to 0+$.  However, since for $|t|\gg 1$,
  \begin{align*}
    L(t,r) = \frac{C_1^2 t^2}{2} \int_{u<1/|t|} u^2\,\ld(\dd u)
    \sim \frac{c C_1^2 \ln |t|}{2},
  \end{align*}
  condition \eqref{e:exp-L} holds only when $c$ is large enough.
  Furthermore, even when \eqref{e:exp-L} holds, the bound in
  \eqref{e:pgn-dtv} decreases to 0 very slowly as $r\to 0$.
  \qed
\end{example}

\section{Multivariate Poisson-Gamma-Normal approximation} 
\label{sec:multivariate}
\subsection{Radial cumulant matching}
In this section, we assume $X\in \Reals^d$ such that in polar
coordinates its $\levy$ measure is
\begin{align*}
  \ld(\dd u,\dd\theta) = \ld(\dd u\gv\theta)\, \nu(\dd\theta), \quad
  \theta\in S,\ u>0,
\end{align*}
where $S=\{\theta\in\Reals^d: \Norm\theta=1\}$, $\nu$ is a finite
measure on $S$, and for each $\theta\in S$, $\ld(\dd u\gv\theta)$ is a
$\levy$ measure on $(0,\infty)$.  For symmetric $X$, $\ld(\dd
u\gv\theta) = \ld(\dd u\gv -\theta)$ and $\nu(\dd\theta) \equiv
\nu(-\dd\theta)$.  Without loss of generality, assume
\begin{gather*}
  \text{$\sppt(\ld)$ is bounded and not contained in a linear
    space of lower dimension}.
\end{gather*}
In particular, the assumption implies $\mean \snorm{X}^p<\infty$ for
all $p>0$.

The so-called radial cumulant matching is as follows.  For each
$\theta\in S$, find $\sigma(\theta)\ge 0$ and a $\levy$ measure
$\cpg(\dd u\gv\theta)$ on $(0,\infty)$, such that, first, for some
$q>2$,
\begin{align*}
  \int u^j \ld(\dd u\gv\theta) = \cf{j=2} \sigma(\theta)^2 + 
  \int u^j \cpg(\dd u\gv\theta) \quad\nu\text{-a.e.}\ \theta
\end{align*}
for all $2\le j<q$ if $X$ is asymmetric, or for all even valued $j\ge
2$ less than $q$ if $X$ is symmetric, and second, for
$\theta\ne\theta'$, if $\ld(\dd u\gv\theta) = \ld(\dd u\gv\theta')$,
then $\sigma(\theta) = \sigma(\theta')$ and $\cpg(\dd u\gv \theta) =
\cpg(\dd u \gv\theta')$.  Let $T$ be an i.d.\ random variable with
\begin{align}
  \che_T(t) = \int\, \Sbr{
    \nth 2\sigma(\theta)^2 \ip t\theta^2 +
    \int (1+\iunit \ip t\theta u- e^{\iunit \ip t\theta u})\,
    \cpg(\dd u\gv\theta)
  }\nu(\dd\theta).
  \label{e:T-chexp}
\end{align}
Then $\mean T = 0$.  If $X$ is asymmetric, then for any $\alpha$ with
$1<|\alpha|< q$,
\begin{align*}
  \kappa_{\alpha, T}
  &=
  \int \theta^\alpha \Sbr{\cf{|\alpha|=2} \sigma(\theta)^2
    + \int u^{|\alpha|}\cpg(\dd u\gv\theta)
  }\nu(\dd\theta)
  =
  \int \theta^\alpha u^{|\alpha|} \ld(\dd u\gv\theta)
  \nu(\dd\theta),
\end{align*}
which is just $\kappa_{\alpha, X}$.  If $X$ is symmetric, the equality
holds for any $\alpha$ with $|\alpha|$ being even and $1<|\alpha|<q$.
On the other hand, if $|\alpha|$ is odd, then $\kappa_{\alpha, X}=0$
and from the construction of $\gamma$, $\kappa_{\alpha, T}=0$.
Therefore, $X$ and $T$ have the same cumulants up to order $q-1$.

The Gaussian component of $T$ has covariance $\int \theta\theta'
\sigma(\theta)^2 \nu(\dd\theta)$, which can be difficult to evaluate.
For normal approximation, the issue can be circumvented by using the
asymptotic of the covariance \cite{cohen:07}.  However, this approach
rules out higher order approximation.  We propose the following
solution.  Since $\ld(\dd u\gv\theta)$ is a $\levy$ measure on
$(0,\infty)$ for each $\theta\in S$, given $\tau>0$, it is possible to
select $r=r(\theta) > 0$ and then set $p=p(\theta)$, $m(r) =
m(r,\theta)$, and $s(r) = s(r,\theta)$ as in Propositions
\ref{prop:cum-match}, \ref{prop:cum-match2}, or
\ref{prop:cum-match-symm}, such that,
letting $\cpg(\dd u\gv\theta) = m(r)\cf{u>0} u^p e^{-u/s(r)}\,\dd u$,
\begin{align*}
  \sigma(\theta)^2= \int \cf{u<r} u^2 \ld(\dd u\gv\theta)
  - \int u^2 \cpg(\dd u\gv\theta)
  = \tau^2>0,
  \quad \text{$\nu$-a.e.\ $\theta\in S$}.
\end{align*}
With this choice of $\cpg(\dd u\gv\theta)$, the Gaussian component is
$N(0, \tau^2 K_\nu)$, where
\begin{align} \label{e:K-theta}
  K_\nu = \int \theta\theta'\nu(\dd\theta)
\end{align}
can be much more manageable.  This is the same matrix identified in
formula (3.17) of \cite{cohen:07}.  By the assumption on $\ld$,
$K_\nu$ is positive definite (p.d.).

There is some flexibility in choosing $\nu$.  Given $w(\theta)$
measurable on $S$ with $0<\ess\inf w \le \ess\sup w < \infty$ under
$\nu$, $\ld(\dd u, \dd\theta)$ can be written as $\tilde\ld(\dd
u\gv\theta) \tilde \nu(\dd \theta)$, where
\begin{align*}
  \tilde\ld(\dd u\gv\theta) = w(\theta) \ld(\dd u\gv \theta), \quad
  \tilde \nu(\dd \theta) = \nu(\dd \theta)/w(\theta).
\end{align*}
If $r$, $p$, $m(r)$ and $s(r)$ are set according to $\tilde\ld(\dd
u\gv\theta)$ instead of $\ld(\dd u\gv\theta)$, then the matrix in
\eqref{e:K-theta} becomes $\int \theta \theta' \tilde\nu(\dd\theta)$.
This allows one to choose $w(\theta)$ to simplify the evaluation of
the matrix.

In this setting, $\tau$ instead of $r$ is the parameter, and $r$, $p$,
$m(r)$, and $s(r)$ are functions of $(\tau,\theta)$.  We denote the
functions by $r_\tau(\theta)$, $p_\tau(\theta)$, $m_\tau(\theta)$ and
$s_\tau(\theta)$, respectively.  Evidently,
\begin{align} \label{e:mv-sigma-tau}
  \int \cf{u<r_\tau(\theta)} u^2 \ld(\dd u\gv\theta)
  - \int_0^\infty m_\tau(\theta) u^{p_\tau(\theta)+2}
  e^{-u/s_\tau(\theta)}\,\dd u
  = \tau^2  \quad\text{$\nu$-a.e.\ $\theta\in S$}.
\end{align}
Additionally,
\begin{align*}
  \ld(\dd u\gv\theta) = \ld(\dd u\gv\theta') \implies
  f(\theta) = f(\theta'), \text{ for } f=r_\tau, p_\tau, m_\tau,
  s_\tau \text{ for all } \tau>0.
\end{align*}

Now define $\levy$ measures
\begin{align}  \label{e:mv-gamma}
  \begin{split}
    \ld_\tau(\dd u,\dd\theta)
    &= \ld_\tau(\dd u\gv\theta)\nu(\dd\theta), \
    \text{with}\
    \ld_\tau(\dd u\gv\theta)
    =\cf{u<r_\tau(\theta)} \ld(\dd u\gv  \theta)\nu(\dd\theta),
    \\
    \cpg_\tau(\dd u, \dd \theta)
    &=
    \cpg_\tau(\dd u\gv\theta)
    \nu(\dd\theta), \ 
    \text{with}\ 
    \cpg_\tau(\dd u\gv\theta)
    =
    m_\tau(\theta)\cf{u>0}
    u^{p_\tau(\theta)} e^{-u/s_\tau(\theta)}\,\dd u.
  \end{split}
\end{align}
Then, for suitable $q\ge 5$, which depends on how $r_\tau$, $p_\tau$,
$m_\tau$, and $s_\tau$ are constructed, 
\begin{align} \label{e:radial-match}
  \int u^j \ld_\tau(\dd u\gv\theta) = \cf{j=2} \tau^2 + 
  \int u^j \cpg_\tau(\dd u\gv\theta),
  \quad\text{$\nu$-a.s.\ $\theta\in S$}
\end{align}
for all $2\le j<q$ if $X$ is asymmetric, or for all even valued $j\ge
2$ less than $q$ if $X$ is symmetric.  Next, decompose $X$ as the sum
of independent i.d.\ random variables $X_\tau$ and $\Delta_\tau$, with
\begin{align*}
  \che_{X_\tau}(t)
  &=
  \int (1+\iunit u\ip t\theta - e^{\iunit u\ip
    t\theta})\, \ld_\tau(\dd u, \dd\theta),
\end{align*}
and $\che_{\Delta_\tau}(t) = \che_X(t) - \che_{X_\tau}(t)$.  
Then approximate $X_\tau$ by
\begin{align*}
  T_\tau= Y_\tau + \tau Z,
  \quad\text{with}\  Z\sim N(0, K_\nu) \text{ independent of
  } Y_r,
\end{align*}
where $Y_\tau$ is i.d.\ with mean 0 and no Gaussian component, and
with $\levy$ measure $\cpg_\tau$.  Finally, $X$ is approximated by
$\Delta_\tau + T_\tau$.

Clearly, in order for the solution to be valid,  $\ld_\tau$ and
$\cpg_\tau$ have to be valid $\levy$ measures.  First, this means
$\cf{u< r_\tau(\theta)}$ and $m_\tau(\theta) \cf{u>0}
u^{p_\tau(\theta)} e^{-u/s_\tau(\theta)}$ must be measurable functions
of $(u,\theta)$.  In many cases, the measurability is not difficult to
verify.  Provided it is established, $\ld_\tau$ immediately is a valid
$\levy$ measure.  On the other hand, since $\int u^2 \cpg_\tau(\dd
u,\dd \theta) \le \int u^2 \ld(\dd u, \dd \theta) < \infty$, by the
comments in Section \ref{ss:assumptions}, $\cpg_\tau$ is a valid
$\levy$ measure and $\mean\snorm{Y_\tau}^2 < \infty$.

The solution shifts the burden of evaluating the normal covariance
to the sampling of $\Delta_\tau$ and $Y_\tau$.  For the latter, the
following statements are true.
\begin{prop} \label{prop:mv}
  1) If
  \begin{align} \label{e:mv-B}
    \int_S B_\tau(\theta)\nu(\dd\theta)<\infty
    \quad\text{with}\
    B_\tau(\theta)
    = \int \cf{u\ge r_\tau(\theta)}\ld(\dd u\gv\theta),
  \end{align}
  then $\Delta_\tau\sim \tilde\zeta_1\tilde\omega_1 + \cdots +
  \tilde\zeta_N \tilde\omega_N - \tilde\mu$, where
  $\{\tilde\omega_i\}$ is a Poisson process on $S$ with $\levy$
  measure $B_\tau\,\dd\nu$, conditional on $\{\tilde\omega_i\}$,
  $\eno{\tilde\zeta} N$ are independent, with $\tilde\zeta_i\sim
  \cf{u\ge r_\tau(\tilde\omega_i)} \ld(\dd u\gv\tilde\omega_i)/
  B_\tau(\tilde\omega_i)$, and
  \begin{align*}
    \tilde\mu = \int \theta u \cf{u\ge r_\tau(\theta)} \ld(\dd
    u\gv\theta) \,\nu(\dd\theta).
  \end{align*}

  2) If $p_\tau(\theta)>-1$ for $\nu$-a.e.\ $\theta$ and
  \begin{align} \label{e:mv-N}
    \int_S N_\tau(\theta) \nu(\dd\theta)<\infty
    \quad\text{ with}\
    N_\tau(\theta) = \Gamma(p_\tau(\theta) + 1) 
    m_\tau(\theta) s_\tau(\theta)^{p_\tau(\theta) +1},
  \end{align}
  then $Y_\tau\sim \zeta_1 \omega_1 + \cdots + \zeta_N\omega_N - \mu$,
  where $\{\omega_i\}$ is a Poisson process on $S$ with $\levy$
  measure $N_\tau\, \dd\nu$, conditional on $\{\omega_i\}$, $\eno\zeta
  N$ are independent, with $\zeta_i \sim \dgamma(p_\tau(\omega_i),
  s_\tau(\omega_i))$, and
  \begin{align*}
    \mu = \int_S \theta (p_\tau(\theta)+1) s_\tau(\theta)
    N_\tau(\theta)\nu(\dd \theta).
  \end{align*}

  3) If $X$ is symmetric, then $\tilde\mu=\mu=0$.

  4) If $\ld$ is direction independent, i.e, $\ld(\dd u\gv\theta) =
  \ld_0(\dd u)$ for $\nu$-a.e.\ $\theta\in S$ for some $\levy$ measure
  $\ld_0$, then given $\tau>0$, $r_\tau$, $m_\tau$, $p_\tau$,
  $s_\tau$, and $N_\tau$ are $\nu$-a.e.\ constant, and $\tilde\mu
  =  \theta_\nu\int u\cf{u\ge r_\tau}\,\ld_0(\dd u)$, $\mu = \theta_\nu
  (p_\tau + 1) s_\tau N_\tau$, where $\theta_\nu = \int_S
  \theta\,\nu(\dd\theta)$.
\end{prop}
\begin{proof}
  1) It is easy to see $\Delta_\tau$ has mean 0 and no Gaussian
  component, and its $\levy$ measure is $\cf{u>r_\tau(\theta)}
  \ld(\dd u\gv\theta) \nu(\dd\theta)$, which by assumption
  \eqref{e:mv-B} has a finite mass.  Then 1) follows from standard
  results on compound Poisson processes \cite{daley:ver:02}.  The
  proof of 2) is similar.  Finally, both 3) and 4) follow from the
  construction of the functions $r_\tau(\theta)$, $p_\tau(\theta)$,
  $s_\tau(\theta)$, and $m_\tau(\theta)$.
\end{proof}

To apply the result, it is desirable that
\begin{align*}
  \ess\sup B_\tau(\theta) < \infty, \quad
  \ess\sup N_\tau(\theta)<\infty \quad\text{under}\ \nu,
\end{align*}
because $\{\tilde \omega_i\}$ and $\{\omega_i\}$ then can be sampled
using the standard thinning method (\cite{devroye:86b}, p.~253--255),
as long as it is easy to sample the Poisson process on $S$ with
$\levy$ measure $C\nu$ for any $C>0$.  In this case, there is no need
to know all $B_\tau(\theta)$ and $N_\tau(\theta)$ beforehand.  Parts
3) and 4) of Proposition \ref{prop:mv} lists two cases where $\tilde
\mu$ and $\mu$ are calculable.  It follows that the quantities are
still calculable if $\ld$ is the sum of a symmetric $\levy$ measure
and a $\levy$ measure $\ld'$ that is ``piecewise'' direction
independent, i.e., $\ld'(\dd u\gv\theta) = \ld_i(\dd u)$ for
$\nu$-a.s.\ $\theta\in S_i$, $i=1,\ldots, n$, where $S_i$ form a
partition of $S$.   However, in other cases, $\tilde\mu$ and $\mu$ may
not have closed form expressions.  This is the limitation alluded in
Introduction.

\begin{example} \label{ex:mv-stable}\rm
  Let $\ld(\dd u, \dd \theta) = \cf{0<u<r_0} u^{-a(\theta)-1}\,
  \dd u\, \nu(\dd\theta)$, where $r_0\in (0,\infty)$ and $a(\theta)\in
  (0,2)$ is a measurable function on $S$, such that
  \begin{align} \label{e:mv-stable-a}
    0< a_0:=\ess\inf a\le \ess\sup a=a_1<2
    \quad\text{under } \nu.
  \end{align}
  Suppose also that $\ld$ is symmetric.  For $a\in (0,2)$, let
  $\pi(a)$ be the unique solution in $(0,\infty)$ to 
  \begin{align*}
    \frac{(p+7)(p+8)}{(p+5)(p+6)}
    = \frac{(6-a)^2}{(4-a)(8-a)}.
  \end{align*}
  Since $\ld(\dd u\gv\theta) = \cf{0<u<r_0} u^{-a(\theta)-1}\,\dd u$,
  in light of Corollary \ref{cor:cum-match2}, define $r_\tau(\theta)$,
  $p_\tau(\theta)$, $s_\tau(\theta)$, and $m_\tau(\theta)$ as follows.
  First, let $p_\tau(\theta) = \pi(a(\theta))$.  Next, provided
  $r_\tau(\theta)\in (0,r_0)$, define
  \begin{align*}
    s_\tau(\theta) = J_1(a(\theta)) r_\tau(\theta),
    \quad
    m_\tau(\theta) = J_2(a(\theta))
    r_\tau(\theta)^{-1-a(\theta) - \pi(a(\theta))}
  \end{align*}
  according to \eqref{e:pgn-s-symm}, where
  \begin{align*}
    J_1(a)
    &=
    \sqrt{\frac{(4-a)}{(\pi(a)+5)(\pi(a)+6)(6-a)}}, 
    \quad
    J_2(a)
    =
    \nth{2\Gamma(\pi(a)+5) (4-a) J_1(a)^{\pi(a)+5}}
  \end{align*}
  are strictly positive and continuous on $(0,2)$.  This yields
  $\int u^2 \ld_\tau(\dd u\gv\theta) = J_0(a(\theta))
  r_\tau(\theta)^{2-a(\theta)}$, where $J_0(a) = 1/(2-a)$, and 
  \begin{align*}
    \int u^2 \cpg_\tau(\dd u\gv\theta) = \Gamma(\pi(a(\theta))+3)
    m_\tau(\theta) s(r)^{\pi(a(\theta))+3} = J_3(a(\theta))
    r_\tau(\theta)^{2-a(\theta)},
  \end{align*}
  where $J_3\in C(0,2)$ is strictly positive.  In particular,
  $0<J_3(a)<J_0(a)$.   Finally, from condition 
  \eqref{e:mv-sigma-tau}, it follows that if $r_\tau(\theta) \in
  (0,r_0)$, then $r_\tau(\theta) = J_4(a(\theta)) \tau^{2 / (2 -
    a(\theta))}$, where
  \begin{align*}
    J_4(a) =[J_0(a) -J_3(a)]^{-1/(2-a)}\in C(0,2).
  \end{align*}
  By assumption \eqref{e:mv-stable-a}, for all small $\tau>0$,
  $\ess\sup J_4(a(\theta)) \tau^{2/(2-a(\theta))} < r_0$, and hence
  all the above definitions are valid.  Since $\pi(a)$ and all
  $J_i(a)$ are continuous functions, $\pi(a(\theta))$ and
  $J_i(a(\theta))$ are measurable functions of $\theta$.  It is then
  easy to see $\cf{u<r_\tau(\theta)}$ is a measurable function of
  $(u,\theta)$ and hence $\ld_\tau$ is a valid $\levy$ measure.
  Likewise, $\cpg_\tau$ is a valid $\levy$ measure.  Consequently, by
  Corollary \ref{cor:cum-match2} and the symmetry of $\ld$,
  \eqref{e:radial-match} is satisfied with $q=10$.

  We consider the sampling of $\Delta_\tau$ and $Y_\tau$
  based on Proposition \ref{prop:mv}.  Given $\tau>0$ small enough,
  by $r_\tau(\theta) = J_4(a(\theta)) \tau^{2/(2-a(\theta))}$ and
  \eqref{e:mv-stable-a},
  \begin{align*}
    \ess\sup B(\theta)
    =
    \ess\sup \int_{r_\tau(\theta)}^{r_0} u^{-a(\theta)-1}\,\dd u
    \le \ess\sup \frac{r_\tau(\theta)^{-a(\theta)}}{a(\theta)} < 
    \infty.
  \end{align*}
  Since $\ld$ is symmetric, then by Proposition \ref{prop:mv},
  $\Delta_\tau=\tilde\zeta_1 \tilde\omega_1 + \cdots + \tilde\zeta_N
  \tilde\omega_N$, where $\{\tilde\omega_i\}$ is a Poisson process
  that can be sampled using the thinning method, while conditioning on
  $\tilde\omega_i$, the probability density of $\tilde\zeta_i$ is 
  proportional to $\cf{r_\tau(\tilde \omega_i) \le u<r_0}
  u^{-a(\tilde\omega_i)-1}$.  On the other hand, by the construction
  of $p_\tau$, $m_\tau$ and $s_\tau$, $N_\tau(\theta) = J(a(\theta))
  r_\tau(\theta)^{-a(\theta)}$ for some continuous $J(a)>0$, so again
  by \eqref{e:mv-stable-a}, $\ess\sup N(\theta)<\infty$.  Then
  $Y_\tau=\zeta_1 \omega_1 + \cdots +\zeta_N \omega_N$,  where
  $\{\omega_i\}$ is a Poisson process that can be sampled using the
  thinning method, while conditioning on $\omega_i$, $\zeta_i\sim
  \dgamma(\pi(a(\omega_i)), s_\tau(\omega_i))$.

  Unfortunately, unlike the univariate case, currently there are no
  computationally efficient methods to sample $\Delta_\tau$ or
  $Y_\tau$, other than sampling individual $(\tilde \omega_i, \tilde
  \zeta_i)$ or $(\omega_i,\zeta_i)$ and then taking the sum of $\tilde
  \zeta_i\tilde \omega_i$ or $\zeta_i\omega_i$.  This raises the issue of
  computational complexity of the PGN approximation.  In the next
  subsection, after obtaining an error bound for the PGN
  approximation, we will come back to the issue and argue that, in
  some cases, comparing to normal approximation \cite{cohen:07}, the
  improvement in error rate is worth the extra computation complexity,
  at least asymptotically.

  Finally, we remark that if $\nu(\dd \theta) = w(\theta)\, \sigma(\dd
  \theta)$, where $\sigma$ is the spherical measure on $S$ and
  $w(\theta)$ is measurable on $S$ with $0<\ess\inf
  w\le \ess\sup w <\infty$ under $\sigma$, then by
  setting $r_\tau(\theta) = J_4(a(\theta)) [\tau^2/w(\theta)]^{1/(2-
    a(\theta))}$ and adjusting $s_\tau(\theta)$ and $m_\tau(\theta)$
  accordingly, we get $K_\nu = \int \theta\theta'\sigma(\dd\theta) =
  I/d$.  With $\sigma(\dd\theta)$ being the new $\nu(\dd\theta)$, the
  sampling of $\Delta_\tau$ and $Y_\tau$ can be done as above.
  \qed
\end{example}

\begin{example} \label{ex:tempered} \rm
  Normal approximation of tempered stable processes is studied in
  detail in \cite{cohen:07}.  For such a process, $\ld(\dd u\gv
  \theta) = \cf{u>0} u^{-a-1} h(u,\theta) \, \dd u$, where for each
  $\theta$, $h(u,\theta)$ is a completely monotone function with
  $h(0+,\theta) =1$ and $h(\infty,\theta)=0$.  A generalized shot
  noise series representation is used in \cite{cohen:07} for normal
  approximation.  Although the resulting covariance of the normal
  distribution in general has no closed form, it is shown that by
  using its asymptotic, the normal distribution still works.

  In the context of PGN approximation to the i.d.\ distribution with
  $\levy$ measure $\ld$, we have to evaluate the covariance
  precisely.   At this point, a solution to the general case has not
  been found.  However, suppose $h(u,\theta)$ can be written as 
  \begin{align*}
    h(u,\theta) = 1-h_1(\theta) u + h_2(u,\theta)u^2,
  \end{align*}
  such that $0\le \ess\inf h_1 \le \ess\sup h_1 < \infty$ under $\nu$
  and $0\le \ess\inf h_2 \le \ess\sup h_2<\infty$ under
  $\ell \times \nu$, where $\ell$ is the Lebesgue measure, then the
  following method can be used.  First, fix $r_0$ such that
  $1-h_1(\theta) r_0\ge 0$ for $\nu$-a.e.\ $\theta$.  Let
  \begin{align*}
    \mu(\dd u\gv \theta) = \cf{0<u<r_0} u^{-a-1} [1-h_1(\theta)u]\,
    \dd u
  \end{align*}
  and $\mu(\dd u,\dd\theta) = \mu(\dd u\gv\theta) \,\nu(\dd \theta)$.
  Decompose $\ld = \mu + \ld_1$, with $\ld_1=\ld - \mu$.  For
  $0<u<r_0$, $\ld_1(\dd u \gv\theta)  = u^{-a-1} [h(u,\theta)
  -1+h_1(\theta)u]\,\dd u\, \nu(\dd\theta) = u^{-a+1} h_2(\theta,u)
  \,\dd u \,\nu(\dd\theta)$, while for $u\ge r_0$, $\ld_1(\dd
  u\gv\theta) = \ld(\dd u\gv\theta)$.  As a result, $\ld_1$ is a
  $\levy$ measure with
  \begin{align*}
    \int \ld_1(\dd u,\dd\theta)
    = \int \nu(\dd\theta) \int_0^{r_0} u^{-a+1} 
    h_2(\theta, u)\,\dd u + \int \cf{u\ge r_0} \ld(\dd
    u,\dd\theta)<\infty
  \end{align*}
  and hence it gives rise to a compound Poisson random variable.  We
  therefore only need to apply PGN approximation to $\mu$.  It is easy
  to compute $\int u^j \mu(\dd u\gv\theta)$ for $j\ge 2$.  Then the
  functions $r_\tau(\theta)$, $p_\tau(\theta)$,  $s_\tau(\theta)$, and
  $m_\tau(\theta)$ can be fixed following Example \ref{ex:mv-stable},
  although the calculation is more tedious due to the extra term
  $h_1(\theta) u$.  \qed 
\end{example}

\subsection{Error bound for approximation} \label{ss:bound-mv}
We next consider how well $X$ is approximated by $\Delta_\tau +
T_\tau$.  For any symmetric p.d.\ matrix $\Sigma$, denote by
$\Sigma^{1/2}$ the unique symmetric p.d.\ matrix whose square is equal
to $\Sigma$.  Given $\tau>0$, let
\begin{align*}
  \Sigma_\tau = \var(X_\tau) = \var(T_\tau), \quad
  \SRM_\tau = \Sigma_\tau^{1/2}, \quad
  \std K = \tau^2 \SRM_\tau^{-1} K_\nu \SRM_\tau^{-1}.
\end{align*}
Let $C_1$ and $C_2$ be the same constants defined in Section
\ref{ss:error-1d}.  Denote, for $z>0$,
\begin{align}
  \begin{array}{c}
  \varrho_\tau(z) 
  =
  \text{smallest eigenvalue of } C_1^2 \SRM_\tau^{-1} M_\tau(z)
  \SRM_\tau^{-1}
  \text{ and }
  \std K
  \\[1ex]
  \displaystyle\text{with}\ \;
  M_\tau(z)
  = 
  \int u^2\theta\theta' \cf{u<r_\tau(\theta)\wedge
    \frac{z}{\snorm{\SRM_\tau^{-1} \theta}}}
  \ld_\tau(\dd u\gv\theta)\,\nu(\dd \theta).
  \end{array}
  \label{e:root-std}
\end{align}
Finally, denote
\begin{align*}
  h(d) = \Flr{d/2} +1.
\end{align*}
\begin{theorem} \label{thm:tv-mv}
  Given $\tau>0$, suppose that under $\nu$,
  \begin{gather}
    \ess\inf p_\tau(\theta)\ge -1, \quad
    \ess\sup \frac{[p_\tau(\theta)+3] s_\tau(\theta)}
    {r_\tau(\theta)}\le 1
    \label{e:mv-p-s-r}
    \\
    \std R:=\ess\sup
    \snorm{\SRM_\tau^{-1}\theta} r_\tau(\theta)
    \le 1, \quad
    \std S:=\ess\sup\,
    [p_\tau(\theta)+ h(d)+1] \snorm{\SRM_\tau^{-1}
      \theta} s_\tau(\theta)
    \le 1
    \label{e:R-S-std}
  \end{gather}
  and for some $q\ge 5$, $\kappa_{\alpha, X_\tau} = \kappa_{\alpha,
    T_\tau}$ for $2\le |\alpha|<q$ and
  \begin{align} \label{e:exp-L-mv}
    \int_0^\infty s^{2q+2h(d)+d-1}
    e^{-\varrho_\tau(1/s) s^2}\,\dd s<\infty.
  \end{align}
  Then
  \begin{align} \label{e:pgn-dtv-mv}
    \dtv(X, \Delta_\tau+T_\tau) \le
    \frac{G(d,q,\tau)}{q!}
    \Sbr{
      \int\snorm{u\SRM_\tau^{-1}\theta}_1^q \ld_\tau(\dd u,
      \dd\theta)
      +
      \int \snorm{u\SRM_\tau^{-1}\theta}_1^q \cpg_\tau(\dd u,
      \dd\theta)
    }
  \end{align}
  where $\snorm x_1$ stands for the $L^1$ norm $|x_1| + \cdots +
  |x_d|$ and 
  \begin{align*}
    G(d,q,\tau) = 
    c(d,q) \sqrt{
      1+\int_{1/\std R}^\infty L_{d,q}(s) e^{-\varrho_\tau(1/s) s^2}
      \,\dd s
    }
  \end{align*}
  with $c(d,q)$ being a constant only depending on $(d,q)$ and
  $L_{d,q}(s)$ a polynomial of order no greater than $2q+2h(d)+d-1$
  whose coefficients are constants only depending on $(d,q)$.
\end{theorem}

\noindent{\it Remark.}
\begin{enumerate}[itemsep=0ex, topsep=.5ex]
\item  Although $\SRM_\tau$ appears in the bound, it is not used
  in the actual construction of $T_\tau$ or $\Delta_\tau$, and
  therefore does not generate a computational problem.
\item A drawback of the bound in Theorem \ref{thm:tv-mv} is that,
  although asymptotically, the error rate can be significantly better
  than normal approximation, the constant coefficients in the bound,
  i.e., $c(d,q)$ and those in $L_{d,q}(s)$, are very large even for
  modest $d$.  Perhaps alternative methods for normal approximation
  (e.g.\ \cite{barbour:chen:05, nourdin:peccati:12,
    chen:goldstein:shao:11}) could be employed to improve these terms,
  or even replace $G(d, q, \tau)$ with a universal constant that only
  depends on $(d,q)$.
\end{enumerate}

In Theorem \ref{thm:tv-mv}, the inequalities in \eqref{e:mv-p-s-r} are
the easiest to establish.  On the other hand, $\std R$ and $\std S$
need more careful treatment as they involve $\SRM_\tau$.  By
\eqref{e:mv-p-s-r}, $\std S$ may be bounded via $\std R$.  The main
technical term in Theorem \ref{thm:tv-mv} is $G(d,q,\tau)$.  The next
result, which will be proved in Section \ref{ss:proof-mv}, provides
some simple criteria to bound $\std R$ and $G(d, q, \tau)$.

\begin{prop} \label{prop:exp-L-mv}
  Under $\nu$, the following statements are  true.

  1) Let $c=c(K_\nu)>0$ be a square root of the smallest eigenvalue of
  $K_\nu$.  Then $\sup_S \snorm{\SRM_\tau^{-1} \theta} \le
  1/(c\tau)$.  If
  \begin{align} \label{e:X-moment-mv2}
    \lim_{\tau\to 0+} \ess\sup r_\tau(\theta)/\tau=0,
  \end{align}
  then $\std R=o(1)$ as $\tau\to 0$.

  2) Given $b>1$ and $q\ge 5$, there is $M=M(b,q,K_\nu)>0$, such that,
  if, in addition to \eqref{e:X-moment-mv2},
  \begin{gather}
    \limsup_{\tau\to 0+}
    \nth{\tau^2}\ess\sup
    \int u^2 \ld_\tau(\dd u\gv\theta)< b,
    \label{e:X-moment-mv1}
    \\
    \liminf_{r\to 0+}
    \nth{r^2\ln(1/r)}
    \ess\inf\int u^2\cf{u<r} \ld(\dd u\gv\theta)
    > M,
    \label{e:X-moment-mv3}
  \end{gather}
  then \eqref{e:exp-L-mv} holds and $G(d,q,\tau)= c(d,q)+o(1)$ as
  $\tau\to 0$.  (Note the variate in \eqref{e:X-moment-mv3} is $r$,
  not $\tau$.)
\end{prop}

\addtocounter{example}{-1}
\begin{example}[Continued]\rm
  Note the assumption in \eqref{e:mv-stable-a}.  By $r_\tau(\theta) =
  J_4(a(\theta)) \tau^{2/(2-a(\theta))}$, \eqref{e:X-moment-mv2} is
  satisfied.  By $\int u^2\ld_\tau(\dd u\gv\theta) = J_0(a(\theta))
  r_\tau(\theta)^{2-a(\theta)} = J(a(\theta)) \tau^2$ for some $J(a)
  \in C(0,2)$, \eqref{e:X-moment-mv1} is satisfied.  Since $\int u^2
  \cf{u<r}\, \ld(\dd u\gv\theta) = r^{2-a(\theta)}/(2-a(\theta))$,
  then \eqref{e:X-moment-mv3} is satisfied no matter the value of $M$.
  Thus we can apply Proposition \ref{prop:exp-L-mv}.  Since
  $p_\tau(\theta) = \pi(a(\theta))>0$ and $(p_\tau(\theta)+3)
  s_\tau(\theta) / r_\tau(\theta) = (p_\tau(\theta)+3)J_1(a)<1$,
  the conditions in
  \eqref{e:mv-p-s-r} are satisfied.  The last condition we need to
  check that for small $\tau>0$, $\std S\le 1$ in 
  \eqref{e:R-S-std}.  However, by \eqref{e:mv-p-s-r} and $\ess\sup
  p_\tau(\theta) = \ess\sup \pi(a(\theta))<\infty$, $\std S = O(\std
  R) = o(1)$.

  We now can apply Theorem \ref{thm:tv-mv}.  By $\snorm x_1 \le
  \sqrt{d} \snorm x$ for $x\in\Reals^d$ and by $\snorm{\SRM_\tau^{-1}
    \theta} \le 1/(c\tau)$, where $c$ is the constant in Proposition
  \ref{prop:exp-L-mv},
  \begin{align*}
    \int\snorm{u\SRM_\tau^{-1}\theta}_1^q \ld_\tau(\dd u,
    \dd\theta)
    &\le
    d^{q/2} \int u^q \snorm{\SRM_\tau^{-1}\theta}_2^q
    \ld_\tau(\dd u, \dd\theta)
    \\
    &\le
    d^{q/2} (c\tau)^{-q} \int u^q \ld_\tau(\dd u,\dd\theta)
    \\
    &=
    (c\tau/\sqrt{d})^{-q}
    \int \frac{r_\tau(\theta)^{q-a(\theta)}}{q-a(\theta)}\,
    \nu(\dd\theta)
    \le
    c' \tau^{(q-2)a_0/(2-a_0)},
  \end{align*}
  where $c'$ is a constant independent of $\tau$.  For $\int \snorm
  {u\SRM_\tau^{-1} \theta}_1^q \cpg_\tau(\dd u,\dd\theta)$, a similar
  bound holds.  Combining these bounds and Proposition
  \ref{prop:exp-L-mv}, $\dtv(X, \Delta_\tau+T_\tau) =
  O(\tau^{(q-2)a_0/(2-a_0)})$, where $q=10$.

  Finally, we compare the computational complexity of the above PGN
  approximation and the normal approximation for $X$  \cite{cohen:07}.
  To make a reasonable comparison, assume $\ld$ is direction
  independent, so that $\ld(\dd u, \dd\theta) = u^{-a-1} \nu(\dd
  \theta)$, where $a\in (0,2)$ is a constant.  Then given $\tau$, both
  approximations use $r_\tau = J_4(a) \tau^{2/(2-a)}$ as the cut-off
  value for jump size and sample $\Delta_\tau$, which involves $N_1\sim
  \dpois(a^{-1}(r_\tau^{-a} - r_0^{-a})\nu(S))$ events.  However, the
  PGN approximation also samples $Y_\tau$, which involves another
  $N_2\sim \dpois(J(a) r_\tau^{-a} \nu(S))$ events.  As $\tau\to 0$,
  $N_2 = O_p(N_1)$, and hence the approximations have the same order
  of complexity.  On the other hand, by Theorem \ref{thm:tv-mv}, the
  $\dtv$ between $X$ and $\Delta_\tau + T_\tau$ is
  $O(\tau^{8a/(2-a)})$, whereas the $\dtv$ between $X$ and its normal
  approximation is $O(\tau^{2a/(2-a)})$.  Therefore, at least
  asymptotically, the PGN approximation has higher but the same order
  of computational complexity as the normal approximation, and the
  extra complexity may lead to significant improvement in error rate
  when $a$ is not too small, e.g., $a>1/8$.
  \qed
\end{example}

\section{Proofs of main results} \label{sec:proof-main}
\subsection{Univariate case} \label{ss:proof-1d}
To prove Theorem \ref{thm:tv}, we can assume that
\begin{align} \label{e:exp-L}
  \int_0^\infty t^{2(q+1)} e^{-2 L(t,r)}\,\dd t<\infty.
\end{align}
Otherwise, $Q_{q+1}=\infty$ and the result is trivial.  We need the
following two lemmas.
\begin{lemma} \label{lemma:rdf}
  1) Let $\xi$ be i.d.\ with $\che_\xi(t) = \int (1+ \iunit t u -
  e^{\iunit t u}) \,\nu(\dd u)$ and $\mean|\xi|^j < \infty$ for
  all $j\ge 1$.  Given $\rx>0$, let $Z\sim N(0,\rx^2)$ be independent
  of $\xi$.  Then $\chf_{\xi + Z}\in\rdf(\Reals)$.

  2) Under condition \eqref{e:exp-L}, $f_{X_r}\in C^q(\Reals)$,
  and for $0\le j\le q$, $f_{X_r}\Sp j(x)\to 0$ as $|x|\toi$.
\end{lemma}

The second lemma is as follows.  Note that it does not require
matching of cumulants.
\begin{lemma} \label{lemma:L1}
  Let $T_r$ be defined as in Theorem \ref{thm:tv}, such that $s(r) <
  1/(p+3)$ and $\sigma(r)>0$.  Fix $\rx>0$.  Given $A$, $B\ge 0$
  with $A+B=1$, let $W$ be an i.d.\ random variable with 
  \begin{align*}
    \che_W(t) = A \che_{X_r}(t) + B \che_{T_r}(t) + \rx^2
    t^2/2.
  \end{align*}
  Let $\xi=W/\nu$, where $\nu = \sqrt{A \kappa_{2, X_r} + B\kappa_{2,
      T_r}}$.  Then $f_\xi\in \rdf(\Reals)$ and for $j\ge 1$,
  \begin{align*}
    \int |f\Sp j_\xi(x)|\,\dd x \le j I_{j-1}(r) +
    I_j(r) + (1+\rx^2/\nu^2) I_{j+1}(r),
  \end{align*}
  where for $j\ge 0$, $I_j(r)\ge 0$ such that
  \begin{align*}
    I_j(r)^2 = \nu^{2j+1} \Sbr{
      \frac{\Gamma(j+1/2)}{2 D(r)^{2j+1}}
      + \int_{1/r}^\infty t^{2j} e^{-2 H(t,r)}\,\dd t
    },
  \end{align*}
  with
  \begin{align*}
    D(r)^2
    &=
    A C_1^2\kappa_{2, X_r}
    + B (C_1^2 C_2 \kappa_{2, Y_r} + \sigma(r)^2),
    \\
    H(t,r)
    &=
    \frac{A C_1^2 t^2}{2}\int_{u<1/|t|} u^2\, \ld(\dd u)
    + \frac{B\sigma(r)^2 t^2}{2}.
  \end{align*}
\end{lemma}

To prove Theorem \ref{thm:tv}, by $\dtv(X, \Delta_r+T_r) =
\dtv(\Delta_r+X_r, \Delta_r + T_r) \le \dtv(X_r, T_r)$, it suffices to
show \eqref{e:pgn-dtv} for $\dtv(X_r, T_r)$.  Let $Z$ and $Z'$ be
i.i.d.\ $N(0,1)$ random variables independent of $X_r$ and $T_r$.  Fix
$\rx>0$.  Letting $h$ be a measurable function with $\snorm
h_\infty\le 1$, our first goal is to bound
\begin{align*}
  \Delta_\rx =\mean[h(X_r+\rx Z)-h(T_r+\rx Z')].
\end{align*}
For $n\ge 2$, we have representations
\begin{align*}
  X_r + \rx Z 
  =
  U_2+\cdots+ U_{n+1}, 
  \quad
  T_r + \rx Z'
  = \cum V n,
\end{align*}
(note the index of $U$ starts at 2), where $U_i$ and $V_j$,
$i,j=1,\ldots, n+1$, are independent i.d.\ random variables with
\begin{align*}
  \che_{U_i}(t) =
  n^{-1} \che_{X_r + \rx Z}(t),\quad
  \che_{V_i}(t) =
  n^{-1} \che_{T_r + \rx Z'}(t).
\end{align*}

For $k=1,\ldots, n+1$, let
\begin{align*}
  W_k
  = 
  \sum_{1\le j<k} V_j + \sum_{k<j\le n+1} U_j,
\end{align*}
and $g_k(x) = \mean h(W_k+x)$.  By $X_r + \rx Z = W_1$ and $T_r + \rx
Z' = W_{n+1}$, it is clear that
\begin{align}
  |\Delta_\rx|
  &=
  |g_1(0) - g_{n+1}(0)|
  \nonumber\\
  &\le
  |\mean[g_1(U_1) - g_{n+1}(V_{n+1})]|
  +|\mean[g_1(U_1) - g_1(0)]|
  + |\mean[g_{n+1}(V_{n+1})-g_{n+1}(0)]|.
  \label{e:E1}
\end{align}

We bound the expectations on the last line separately.  By $W_k + V_k
= W_{k+1} + U_{k+1}$,
\begin{align*}
  h(W_1+U_1) - h(W_{n+1}+V_{n+1})
  =
  \sum_{k=1}^{n+1} [h(W_k + U_k) - h(W_k + V_k)].
\end{align*}
By independence, $\mean h(W_k + U_k) = \mean g_k(U_k)$ and $\mean
h(W_k + V_k) = \mean g_k(V_k)$.  Therefore, taking expectation on both
sides of the displayed identity yields
\begin{align}
  \mean[g_1(U_1) - g_{n+1}(V_{n+1})]
  =
  \sum_{k=1}^{n+1} \mean[g_k(U_k) - g_k(V_k)].
  \label{e:E2}
\end{align}
Denote $\nu=\kappa_{2,X_r}^{1/2}$.  Let $\xi_k = W_k/\nu$.  By Lemma
\ref{lemma:rdf}, $f_{\xi_k}\in \rdf(\Reals)$.  As a result,
\begin{align*}
  g_k(x) &= \int h(\nu u + x) f_{\xi_k}(u)\,\dd u
  = \int h(\nu u) f_{\xi_k}(u-x/\nu)\,\dd u
\end{align*}
is smooth.  By Taylor expansion around 0,
\begin{align*}
  g_k(U_k) - g_k(V_k) = \sum_{j=1}^{q-1}
  \frac{g_k\Sp j(0)}{j!} (U_k^j -V_k^j) +
  \nth{q!} [g_k\Sp q(\theta(U_k) U_k) U_k^q
  - g_k\Sp q(\theta(V_k) V_k) V_k^q],
\end{align*}
where $\theta(x)\in [0,1]$.  By assumption, $\kappa_{j, X_r} =
\kappa_{j, T_r}$ for $1\le j < q$.  Since $\kappa_{j, U_k} = n^{-1}
\kappa_{j, X_r+\rx Z} = n^{-1}(\kappa_{j, X_r} + \rx^2 \cf{j=2})$, and
likewise $\kappa_{j, V_k} = n^{-1} (\kappa_{j, T_r} + \rx^2
\cf{j=2})$, then $\kappa_{j,U_k} = \kappa_{j, V_k}$ for $1\le j<q$.
As a result, $\mean U_k^j = \mean V_k^j$ for $1\le j<q$ and hence
\begin{align*}
  \mean[g_k(U_k) - g_k(V_k)]
  &=
  \nth{q!} \mean [g_k\Sp q(\theta(U_k) V_k) U_k^q
    - g_k\Sp q(\theta(V_k) V_k) V_k^q]
\end{align*}
giving
\begin{align}
  |\mean[g_k(U_k) - g_k(V_k)]|
  &\le
  \frac{\snorm{g_k\Sp q}_\infty}{q!}
  [\mean |U_k|^q + \mean |V_k|^q]. 
  \label{e:E-g}
\end{align}
Since $g_k\Sp q(x) = (-\nu)^{-q} \int h(\nu u) f_{\xi_k}\Sp
q(u-x/\nu)\,\dd u$, then
\begin{align} \label{e:g-derivative}
  \snorm{g_k\Sp q}_\infty
  \le
  \nu^{-q} \int |f_{\xi_k}\Sp q(u)|\,\dd u.
\end{align}
Because
\begin{align*}
  \che_{W_k}(t)
  &=
  (k-1) \che_{V_1}(t) + (n+1-k)\che_{U_1}(t) \\
  &
  = \frac{n+1-k}{n} \che_{X_r}(t) 
  +
  \frac{k-1}{n} \che_{T_r}(t)
  + \frac{\rx^2 t^2}{2},
\end{align*}
we can apply Lemma \ref{lemma:L1} with $\nu^2 = \kappa_{2,X_r} =
\kappa_{2,T_r}$, $A=(n+1-k)/n$ and $B=(k-1)/n$ therein.  By
definition of $D(r)$ and $H(t,r)$ in Lemma \ref{lemma:L1},
\begin{align*}
  D(r)^2
  &=
  A C_1^2\nu^2
  + B (C_1^2 C_2 \kappa_{2, Y_r} + \sigma(r)^2)
  \ge
  A C_1^2\nu^2
  + B C_1^2 C_2 (\kappa_{2, Y_r} + \sigma(r)^2)
  \ge C_1^2 C_2 \nu^2
\end{align*}
and
\begin{align*}
  H(t,r)
  &=
  \frac{A C_1^2 t^2}{2}\int_{u<1/|t|} u^2\, \ld(\dd u)
  +\frac{B\sigma(r)^2 t^2}{2}
  \\
  &\ge
  \frac{(A+B) t^2}{2}
  \min\Cbr{
    C_1^2 \int_{u<1/|t|} u^2\, \ld(\dd u), \ \sigma(r)^2
  } 
  =
  L(t, r).
\end{align*}
By definition of $Q_j(r)$ in Theorem \ref{thm:tv} and definition of
$I_j(r)$ in Lemma \ref{lemma:L1}, $I_j(r)^2\le Q_j(r)^2$.   By
condition \eqref{e:exp-L}, $Q_j(r)^2<\infty$ for $0\le j\le q+1$.
Thus \eqref{e:g-derivative} and Lemma \ref{lemma:L1} give
\begin{align*}
  \snorm{g_k\Sp q}_\infty
  &\le
  \nu^{-q} \Sbr{
    q Q_{q-1}(r) + Q_q(r) + (1+\rx^2/\nu^2)Q_{q+1}(r)
  }
  := M_\rx < \infty.
\end{align*}
Since $M_\rx$ is independent of $k$, by \eqref{e:E2} and
\eqref{e:E-g},
\begin{align*}
  |\mean g_1(U_1)-\mean g_{n+1}(V_{n+1})|
  &\le
  \sum_{k=1}^{n+1} |\mean[g_k(U_k) - g_k(V_k)]| 
  \le
  \frac{M_\rx}{q!} \sum_{k=1}^{n+1} (\mean |U_k|^q + \mean |V_k|^q).
\end{align*}
Since the $\levy$ measure of $X_r$ has bounded support, $\mean|X_r +
\rx Z|^q<\infty$.  Meanwhile, from \eqref{e:cum-X-Y}, $\mean
|Y_r+\rx Z|^q<\infty$.  Then by Lemma 3.1 in \cite{asmussen:01},
\begin{align*}
  \sum_{k=1}^{n+1} \mean |U_k|^q \to |\kappa|_{q, X_r + \rx Z}
  = |\kappa|_{q, X_r}, \quad 
  \sum_{k=1}^{n+1} \mean |V_k|^q \to |\kappa|_{q, T_r + \rx Z'}
  = |\kappa|_{q, Y_r}.
\end{align*}
As a result,
\begin{align} \label{e:E3}
  \limsup_{n\toi} |\mean g_1(U_1)- \mean g_{n+1}(V_{n+1})| \le
  \frac{M_\rx}{q!}(|\kappa|_{q, X_r} + |\kappa|_{q, Y_r}).
\end{align}

On the other hand, $|\mean[g_1(U_1)-g_1(0)]| \le \snorm{g_1'}_\infty
\mean|U_1|$.  Since $\mean U_1^2 = \var(U_1) = \var(X_r+\rx Z)/n$, by
Cauchy-Schwartz inequality, as $n\toi$, $\mean|U_1|\to 0$.  As in
\eqref{e:g-derivative}, $\snorm{g_1'}_\infty \le \nu^{-1} \int
|f_{\xi_1}'(u)|\,\dd u$.  By $f_{\xi_1}\in \rdf(\Reals)$,
$\snorm{g_1'}_\infty < \infty$.  Since $g_1(x) = \mean h(X_r+\rx Z+x)$
is a function independent of $n$, it follows that
$\mean[g_1(U_1)-g_1(0)]\to 0$ as $n\toi$.  Likewise,
$\mean[g_{n+1}(V_{n+1})-g_{n+1}(0)]\to 0$.  Together with \eqref{e:E1}
and \eqref{e:E3}, this implies
\begin{align*}
  |\mean h(X_r+\rx Z)- \mean h(T_r + \rx Z')| \le
  \frac{M_\rx}{q!}(|\kappa|_{q, X_r} + |\kappa|_{q, Y_r}).
\end{align*}

Let $G\subset \Reals$ be the union of a finite number of $(a_i,b_i)$
and $h(x)=\cf{x\in G}$.  By 2) of Lemma \ref{lemma:rdf}, $\pr\{X_r =
a_i$ or $b_i$, some $i\}=0$.  Then $h(X_r + \rx Z)- h(X_r)\to 0$ a.s.\
as $\rx\to 0+$.  On the other hand, since $T_r$ is the sum of $Y_r$
and an independent nonzero Gaussian random variable, by 1) of Lemma
\ref{lemma:rdf}, $f_{T_r}\in \rdf(\Reals)$.  As a result, $h(T_r + \rx
Z')- h(T_r)\to 0$ a.s.\ as $\rx\to 0+$.  Also, $M_\rx \to
M:=\nu^{-q}[q Q_{q-1}(r) + Q_q(r) + Q_{q+1}(r)]$.  Thus, by dominated
convergence,
\begin{align*}
  |\pr\{X_r\in G\}-\pr\{T_r\in G\}|
  \le \frac{M}{q!}(|\kappa|_{q, X_r} + |\kappa|_{q, Y_r}).
\end{align*}

Let $A\subset\Reals$ be measurable.  Given $\delta>0$, there is $R>0$,
such that, letting $B=A\cap (-R, R)$, $\pr\{X_r\in A\setminus B\} +
\pr\{T_r\in A\setminus  B\}<\delta$.   Then there is an open $G
\supset B$ with $\ell(G\setminus B)<\delta$, where $\ell$ is the
Lebesgue measure.  $G$ is the union of at most countably many
disjoint open intervals $(a_i, b_i)$.  Let $G_k = \cup_{i=1}^k
(a_i, b_i)$.  Then
\begin{align*}
  &
  |\pr\{X_r\in A\}-\pr\{T_r\in A\}|
  \\
  &\le
  |\pr\{X_r\in G_k\} - \pr\{T_r\in G_k\}| 
  + \pr\{X_r\in G\setminus G_k\} + \pr\{T_r\in G\setminus G_k\}
  \\
  &\quad
  + \pr\{X_r\in A\setminus G\} + \pr\{T_r\in A\setminus G\}
  + \pr\{X_r\in G\setminus A\} + \pr\{T_r\in G\setminus A\}
  \\
  &\le
  \frac{M}{q!}(|\kappa|_{q, X_r} + |\kappa|_{q, Y_r})
  + \pr\{X_r\in G\setminus G_k\} + \pr\{T_r\in G\setminus G_k\}
  \\
  &\quad
  + \pr\{X_r\in A\setminus B\} + \pr\{T_r\in A\setminus B\}
  +  (\snorm{f_{X_r}}_\infty + \snorm{f_{T_r}}_\infty)
  \,\ell(G\setminus B)
  \\
  &\le
  \frac{M}{q!}(|\kappa|_{q, X_r} + |\kappa|_{q, Y_r})
  + \pr\{X_r\in G\setminus G_k\} + \pr\{T_r\in G\setminus G_k\}
  \\
  &\quad
  +(1+\snorm{f_{X_r}}_\infty + \snorm{f_{T_r}}_\infty)\delta.
\end{align*}
By Lemma \ref{lemma:rdf}, $\snorm{f_{X_r}}_\infty +
\snorm{f_{T_r}}_\infty<\infty$.  Then, letting $k\toi$ followed by
$\delta\to 0$ yields
\begin{align*}
  |\pr\{X_r\in A\}-\pr\{T_r\in A\}]| \le
  \frac{M}{q!} (|\kappa|_{q, X_r} + |\kappa|_{q, Y_r}).
\end{align*}
This completes the proof of Theorem \ref{thm:tv}.

\subsection{Multivariate case} \label{ss:proof-mv}
We shall prove Theorem \ref{thm:tv-mv} by standardizing the random
variables involved.  First, we have the following simple result.
\begin{lemma} \label{lemma:standard}
  Let $X$ be i.d.\ with $\levy$ measure $\ld(\dd u, \dd\theta)$ in
  polar coordinates and $\SRM$ be a nonsingular matrix.  Then
  $\SRM^{-1} X$ has $\levy$ measure $\ld(J_\SRM^{-1}(\cdot))$ and
  variance $\SRM^{-1}\var(X) \SRM^{-1}$, where
  \begin{align*}
    J_\SRM(u,\theta) = \Grp{
      \snorm{\SRM^{-1}\theta} u, \,
      \frac{\SRM^{-1}\theta}{\snorm{\SRM^{-1}\theta}}
    }, \quad (u,\theta)\in (0,\infty)\times S.
  \end{align*}
  Furthermore, $J^{-1}_\SRM(u,\theta) = J_{\SRM^{-1}}(u,\theta)
  = (\snorm{\SRM\theta} u,\, \SRM\theta/\snorm{\SRM\theta})$.
\end{lemma}
\begin{proof}
  Denote $h(z) = 1 + z - e^z$.  For any $t\in\Reals^d$,
  \begin{align*}
    \che_{\SRM^{-1} X}(t)
    = \che_X (\SRM^{-1} t)
    &=
    \int h(\iunit \ip {\SRM^{-1}t}\theta u)
    \,\ld(\dd u,\, \dd\theta)
    \\
    &=
    \int h(\iunit\ip t{\SRM^{-1}\theta} u) \,\ld(\dd u,\, \dd\theta)
    =
    \int h(\iunit\ip t\omega v)\,\ld(\dd u,\, \dd\theta),
  \end{align*}
  where $(v,\omega) = J_\SRM(u,\theta)$.  Then the lemma easily
  follows.
\end{proof}

Recall $\SRM_\tau = \Sigma_\tau^{1/2}$ and $\std K = \tau^2
\SRM_\tau^{-1} K_\nu \SRM_\tau^{-1}$, where $\Sigma_\tau=\var(X_\tau)
= \var(T_\tau)$.  Let
\begin{align*}
  \std X = \SRM_\tau^{-1} X_\tau,
  \quad
  \std T = \SRM_\tau^{-1} T_\tau,
  \quad
  \std Y = \SRM_\tau^{-1} Y_\tau.
\end{align*}
Then $\std X$ and $\std T$ are standardized, i.e., $\mean\std X =
\mean \std T=0$ and $\var(\std X) = \var(\std T) = I$.  It is easy to
check that if $\kappa_{\alpha, X_\tau} = \kappa_{\alpha, T_\tau}$ for
$2\le|\alpha|<q$, then $\kappa_{\alpha, \std X} = \kappa_{\alpha, \std
  T}$ for $2\le |\alpha|<q$.  By Lemma \ref{lemma:standard}, $\std X$
and $\std Y$ have $\levy$ measures $\std\ld(\dd u, \dd\theta) =
\ld_\tau(\dd v, \dd\omega)$ and $\std\cpg(\dd u, \dd\theta) =
\cpg_\tau(\dd v, \dd\omega)$, respectively, where $(v,\omega) =
J_{\SRM_\tau}^{-1}(u,\theta)$, and $\std T = \std Y + \std Z$, where
$\std Z\sim N(0, \std K)$ is independent of $\std Y$.

\begin{lemma} \label{lemma:mv-moments}
  Under conditions \eqref{e:mv-p-s-r} and \eqref{e:R-S-std}, $\mean
  \snorm{\std X}^a < \infty$ and $\mean \snorm{\std Y}^a < \infty$ for
  any $a>0$.
\end{lemma}

The next lemma is analogous to Lemma \ref{lemma:rdf}.
\begin{lemma} \label{lemma:mv-smooth}
  Given $\tau>0$, the following statements are true.

  1) Under condition \eqref{e:mv-p-s-r} and \eqref{e:R-S-std}, for any
  $a\ge 0$ and $b\ge 0$ with $a+b>0$ and $\rx>0$, if $\xi$ is i.d.\
  with $\che_\xi = a \che_{\std X} + b \che_{\std Y}$ and $Z\sim N(0,
  \rx^2 I)$ is independent of $\xi$, then $\chf_{\xi + Z}\in
  \rdf(\Reals^d)$.
  
  2) Under condition \eqref{e:exp-L-mv}, $f_{\std X} \in
  C^q(\Reals^d)$ and for each $|\alpha|\le q$, $f\Sp\alpha_{\std
    X}(x)\to 0$ as $|x|\toi$.
\end{lemma}

The next lemma is analogous to Lemma \ref{lemma:L1}.
\begin{lemma} \label{lemma:L1-mv-2}
  Given $\rx>0$ and $A$, $B\ge 0$ with $A+B=1$, let $\xi$ be an i.d.\
  random variable with
  \begin{align*}
    \che_\xi(t) = A \che_{\std X}(t) + B \che_{\std T}(t) +
    \rx^2\snorm t^2/2.
  \end{align*}
  Then, under condition \eqref{e:R-S-std}, for
  $m\ge 3$,
  \begin{align*}
    \max_{|\alpha|=m} \int |f_\xi\Sp \alpha(x)|\,\dd x
    \le (1+\rx^2/d)^{h(d)} c(d,m)
    \sqrt{
      1+\int_{1/\std R}^\infty L_{d,m}(s) e^{-\varrho_\tau(1/s) s^2}
      \,\dd s
    },
  \end{align*}
  where $c(d,m)$ is a constant only depending on $(d,m)$ and
  $L_{d,m}(s)$ a polynomial of order no greater than $2m + 2h(d)+d-1$
  whose coefficients are constants only depending on $(d,m)$.
\end{lemma}

The proof of Theorem \ref{thm:tv-mv} is similar to the one for Theorem
\ref{thm:tv}, so we will only give its sketch.  By $\dtv(X,
\Delta_\tau + T_\tau) \le \dtv(X_\tau, T_\tau)=\dtv(\std X, \std T)$,
it suffices to show \eqref{e:pgn-dtv-mv} for $\dtv(\std X, \std T)$.
Let $Z$ and $Z'\in\Reals^d$ be i.i.d.\ $N(0,I)$ random variables
independent of $\std X$ and $\std T$.  Given $\rx>0$, for $n\ge 2$,
\begin{align*}
  \std X + \rx Z = U_2+\cdots+U_{n+1}, \quad
  \std T + \rx Z' = \cum V n,
\end{align*}
where $U_i$ and $V_j$ are independent and i.d.\ with
$\che_{U_i} = n^{-1} \che_{\std X_r + \rx Z}$ and $\che_{V_j} =
n^{-1} \che_{\std T_r + \rx Z'}$.  Let $W_k = \cum V {k-1} +
U_{k+1} + \cdots + U_{n+1}$.  Given a measurable function $h$ with
$\snorm h_\infty\le 1$, let $g_k(x) = \mean h(W_k+x)$.  Then
\begin{align*}
  &
  |\mean[h(\std X+\rx Z)-h(\std Y + \rx Z')]|
  \\
  &\le
  \sum_{k=1}^{n+1} |\mean[g_k(U_k) - g_k(V_k)]|
  + |\mean[g_1(U_1) - g_1(0)]| + |\mean[g_{n+1}(V_{n+1}) - g_{n+1}(0)]|
\end{align*}
Since $g_k(x) = \int h(u) f_{W_k}(u-x)\,\dd u$ and $\snorm h_\infty\le
1$, by Lemma \ref{lemma:mv-smooth}, $g_k\in C^\infty$.  Then
\begin{align*}
  g_k(U_k) - g_k(V_k) = \sum_{|\alpha|<q}
  \frac{g_k\Sp\alpha(0)}{\alpha !} (U_k^\alpha -V_k^\alpha) +
  \sum_{|\alpha|=q} \nth{\alpha !}
  [g_k\Sp\alpha(\theta(U_k) U_k) U_k^\alpha
  - g_k\Sp\alpha(\theta(V_k) V_k) V_k^\alpha],
\end{align*}
where $\theta(x)\in [0,1]$.  Since $\mean U_k^\alpha = \mean
V_k^\alpha$ for $2\le |\alpha|<q$, then
\begin{align*}
  |\mean[g_k(U_k) - g_k(V_k)]|
  \le \sum_{|\alpha|=q} \frac{\snorm {g_k\Sp\alpha}_\infty}{\alpha!}
  [\mean|U_k^\alpha| + \mean|V_k^\alpha|].
\end{align*}
By Lemma \ref{lemma:L1-mv-2}, for $|\alpha|=q$,
\begin{align*}
  \snorm{g_k\Sp\alpha}_\infty
  \le \int |f\Sp\alpha_{W_k}(u)|\,\dd u
  \le (1+\rx^2/d)^{h(d)} G(d,q, \tau),
\end{align*}
which is finite by \eqref{e:exp-L-mv}.  Note $|u^\alpha|\le \snorm
u^q$.  Then by Lemma \ref{lemma:mv-moments} and Proposition
\ref{prop:AR}, as $n\toi$,
\begin{align*}
  \sum_{k=1}^{n+1}
  \mean|U_k\Sp\alpha|\to |\kappa|_{\alpha, \std X},
  \quad
  \sum_{k=1}^{n+1}
  \mean|V_k\Sp\alpha|\to |\kappa|_{\alpha, \std Y}.
\end{align*}
From here, an argument similar to that for Theorem \ref{thm:tv} leads
to
\begin{align*}
  \dtv(\std X, \std T)
  \le
  G(d,q,\tau)
  \sum_{|\alpha|=q}
  \frac{|\kappa|_{\alpha, \std X} + |\kappa|_{\alpha, \std Y}}
  {\alpha!}.
\end{align*}
Since
\begin{align*}
  \sum_{|\alpha|=q}
  \frac{|\kappa|_{\alpha, \std X}}{\alpha!}
  &=
  \sum_{|\alpha|=q} \nth{\alpha!}
  \int |(u\theta)^\alpha|\,\std\ld(\dd u, \dd\theta)
  \\
  &=
  \sum_{|\alpha|=q} \nth{\alpha!}
  \int |(u\SRM_\tau^{-1}\theta)^\alpha|
  \,\ld_\tau(\dd u, \dd\theta)
  = \nth{q!} 
  \int
  \snorm{(u\SRM_\tau^{-1}\theta)}_1^q \,\ld_\tau(\dd u, \dd\theta)
\end{align*}
and a similar expression holds for $\sum_{|\alpha|=q}
|\kappa|_{\alpha, \std T}/\alpha!$, the proof of Theorem
\ref{thm:tv-mv} is then complete.

\begin{proof}[Proof of Proposition \ref{prop:exp-L-mv}]
  1) By $\tau^2 K_\nu \le \Sigma_\tau$, the smallest eigenvalue
  of $\Sigma_\tau$ is at least $\tau^2 c(K_\nu)^2$, yielding
  the first assertion and $\std R = O(\ess\sup r_\tau(\theta)/\tau)$.
  Then by \eqref{e:X-moment-mv2}, $\std R\to 0$ as $\tau\to 0+$.

  2) By condition \eqref{e:X-moment-mv1}, for small $\tau>0$,
  $\Sigma_\tau = \int \theta\theta'\,\nu(\dd\theta) \int 
  u^2\ld_\tau(\dd u\gv\theta) \le b\tau^2 K_\nu$ and hence
  \begin{align} \label{e:X-moment-root}
    \tau^2 K_\nu \ge c_1\Sigma_\tau,
  \end{align}
  where $c_1 = 1/b>0$.  Denote
  \begin{align*}
    g_\tau(s,\theta) = \int u^2 \cf{
      u<r_\tau(\theta)\wedge \nth{\snorm{\SRM_\tau^{-1} \theta} s}
    }\,\ld(\dd u\gv\theta).
  \end{align*}
  Then $M_\tau(1/s) = \int \theta\theta' g_\tau(s,\theta)\,\nu(\dd
  \theta)$.  If $1/\std R < s < \nth{\snorm{\SRM_\tau ^{-1}\theta}
    r_\tau(\theta)}$, then by definition of $r_\tau(\theta)$,
  \begin{align*}
    g_\tau(s,\theta)
    =
    \int u^2 \cf{u<r_\tau(\theta)}\,\ld(\dd u\gv\theta)
    = \tau^2.
  \end{align*}
  If $s\ge \nth{\snorm{\SRM_\tau ^{-1}\theta}
    r_\tau(\theta)}$, then by $\snorm{\SRM_\tau^{-1}\theta}\le
  1/(c\tau)$ and condition \eqref{e:X-moment-mv3}, with $M$ to be
  determined,
  \begin{align*}
    g_\tau(s,\theta)
    &=
    \int u^2 \cf{
      u<\nth{\snorm{\SRM_\tau^{-1} \theta} s}
    }\,\ld(\dd u\gv\theta)
    \ge
    \int u^2 \cf{
      u<\frac{c\tau} {s}
    }\,\ld(\dd u\gv\theta)
    \ge
    \frac{c^2 \tau^2 M}{s^2}
    \ln \frac{s}{c\tau},
  \end{align*}
  Thus, for $s>1/\std R$, $g_\tau(s,\theta) \ge \tau^2 F_\tau(s)$,
  where
  \begin{align*}
    F_\tau(s) = \min\Cbr{1,\,\frac{c^2 M}{s^2} \ln\frac{s}{c\tau}},
  \end{align*}
  and as a result, by \eqref{e:X-moment-root}, $M_\tau(1/s) \ge \tau^2
  F_\tau(s) K_\nu \ge c_1 F_\tau(s) \Sigma_\tau$, giving
  $\SRM_\tau^{-1} M_\tau(1/s) \SRM_\tau^{-1} \ge c_1 F_\tau(s)$.
  Also, by \eqref{e:X-moment-root}, $\std K = \tau^2 \SRM_\tau^{-1}
  K_\nu \SRM_\tau^{-1} \ge c_1 I$.  Thus, there is $c_2 = c_2(b,
  K_\nu)>0$, such that for $\tau>0$ small enough and $s>1/\std R$,
  \begin{align*}
    \varrho_\tau(1/s)s^2 \ge c_2 F_\tau(s) s^2
    = 
    c_2 \min\Cbr{s^2, c^2 M\ln \frac{s}{c\tau}}.
  \end{align*}
  If $M>0$ is large enough so that $c_2 c^2 M> 2 q + 2h(d) + d$, then
  \begin{align*}
    \int_{1/\std R}^\infty s^{2q + 2h(d)+d-1} e^{-\varrho_\tau(1/s)
      s^2}\,\dd s
    \le \int_{1/\std R}^\infty s^{2q + 2h(d)+d-1} 
    \Cbr{e^{-c_2 s^2} + \Grp{\frac{s}{c \tau}}^{-c_2 c^2 M}}\,
    \dd s
  \end{align*}
  is $o(1)$ as $\tau\to 0$.  It then easily follows that
  $G(d,q,\tau) = c(d,q) + o(1)$.
\end{proof}

\section{Proofs of auxiliary results} \label{sec:proof-aux}
\subsection{Proposition \ref{prop:AR}}
Given $0<a<b<\infty$, let $0\le h(u)\le 1$ be a continuous
function with compact support in $\Reals^d \setminus\{0\}$, such
that $h(u)=1$ if $\snorm u\in [a,b]$.  As pointed out in Section
\ref{ss:assumptions},
\begin{align*}
  &
  \lim_{n\toi} n\mean[h(U_n) f(U_n)]
  =\int h(u) f(u)\,\ld(\dd u)<\infty, \\
  &
  \lim_{n\toi} n\mean[h(U_n) g(\snorm{U_n})]
  =\int h(u) g(\snorm u)\,\ld(\dd u)<\infty.
\end{align*}
From the second line and monotone convergence it follows that
\begin{align*}
  \liminf_{n\toi} n \mean[g(\snorm{U_n})]
  \ge\int g(\snorm u)\,\ld(\dd u).
\end{align*}
Clearly, $f\in L^1(\ld)$ if $g(\snorm u)\in L^1(\ld)$.  Furthermore,
\begin{align*}
  |n\mean[f(U_n)]-n\mean[h(U_n) f(U_n)]|
  &\le
  n\mean [(1-h(U_n)) g(\snorm{U_n})]
  \\
  &\le
  n\mean [g(\snorm{U_n})\cf{\snorm{U_n}\le a \text{ or }
    \snorm{U_n}\ge b}].
\end{align*}
Therefore, to prove the lemma, it suffices to show
\begin{align}
  &
  \limsup_{n\toi}
  n\mean [g(\snorm{U_n})\cf{\snorm{U_n}\le a}]
  \to 0, \quad\text{as } a\to 0+,
  \label{e:limit-1}
  \\
  &
  \limsup_{n\toi}
  n\mean [g(\snorm{U_n})\cf{\snorm{U_n}\ge b}]
  \to 0, \quad\text{as } b\toi.
  \label{e:limit-2}
\end{align}
  
Let $m(a) = \sup_{0<\snorm t\le a} g(t)/t^2$.   For $n\ge 1$, since
$\var(U_n) = \var(X)/n$ and $\mean U_n = \mean X/n$,
\begin{align*}
  n \mean[|g(\snorm{U_n})\cf{\snorm{U_n}\le a}]
  &\le
  m(a) n\mean[\snorm{U_n}^2]
  \\
  &=
  m(a) n[\tr(\var(U_n)) + \snorm{\mean U_n}^2]
  \\
  &= 
  m(a) [\tr(\var(X)) + \snorm{\mean X}^2/n].
\end{align*}
Since $m(a)\to 0$, then \eqref{e:limit-1} follows.  Next, since $g$ is
continuous and non-decreasing,
\begin{align}
  n\mean [g(\snorm{U_n})\cf{\snorm{U_n}\ge b}]
  &=
  \int_0^\infty n \pr\{g(\snorm{U_n})\ge s, \snorm{U_n}\ge b\}\,\dd s
  \nonumber
  \\
  &=
  \int_0^\infty n \pr\{\snorm{U_n}\ge g^*(s)\vee b\}\,\dd s,
  \label{e:g-truncate}
\end{align}
where $g^*(s) = \inf\{x: g(x)\ge s\}$.  By Markov inequality, 
\begin{align*}
  n \pr\{\snorm{U_n}\ge x\} \le \frac{n \mean\snorm{U_n}^2}{x^2}
  = \frac{\tr(\var(X)) + \snorm{\mean X}^2/n}{x^2}.
\end{align*}
Therefore, for large $x>0$, $n\pr\{\snorm{U_n}\ge x\}\le 1$.  Since
$c_0 t\le 1-(1-t/n)^n$ for $t\in [0,1]$, where $c_0>0$ is a universal
constant, then
\begin{align} \label{e:joint}
  c_0 n \pr\{\snorm{U_n}\ge x\}
  \le 1-(1-\pr\{\snorm{U_n}\ge x\})^n
  = \pr\Cbr{\max_{1\le k\le n} \snorm{U_{n,k}}\ge x}
\end{align}
where $U_{n,k}$ are i.i.d.\ $\sim U_n$.

First, suppose $X$ is asymmetric.  Let $V_{n1}$, \ldots, $V_{n n}\sim
U_n$ be another set of i.i.d.\ random variables which are also
independent of $U_{n k}$.  Fix $\delta\in (0,1)$.  By setting $b>0$
even larger, for all $x\ge b$,
\begin{align*}
  \pr\Cbr{\max_{1\le k\le n} \snorm{V_{n k}}<\delta x}
  \ge 1 - n\pr\{\snorm{U_n}\ge \delta x\} \ge 1/2.
\end{align*}
As a result,
\begin{align*}
  c_0 n \pr\{\snorm{U_n}\ge x\}
  &\le
  2 \pr\Cbr{\max_{1\le k\le n} \snorm{U_{n k}}\ge x}
  \pr\Cbr{\max_{1\le k\le n} \snorm{V_{n k}}< \delta x}
  \nonumber \\
  &\le 
  2 \pr\Cbr{\max_{1\le k\le n} \snorm{U_{n k} - V_{n k}}\ge
    (1-\delta) x
  }
  \le
  4 \pr\Cbr{\snorm{X'}\ge (1-\delta) x
  }
\end{align*}
where the last line is due to the symmetry of $U_{n k} - V_{n k}$ and
\cite{ledoux:91}, Proposition 2.3.  Thus, by
\eqref{e:g-truncate},
\begin{align*}
  n \mean[g(\snorm {U_n}) \cf{\snorm {U_n}\ge b}]
  &\le 
  \frac{4}{c_0} \int_0^\infty
  \pr\Cbr{\frac{\snorm{ X' }}{1-\delta}\ge g^*(s)\vee b
  }\, \dd s
  \\
  &=
  \frac{4}{c_0} \int_0^\infty
  \pr\Cbr{g\Grp{\frac{\snorm{ X' }}{1-\delta}}\ge s, \
    \frac{\snorm{ X' }}{1-\delta}\ge b
  }\, \dd s
  \\
  &=
  \frac{4}{c_0}\mean\Sbr{
    g\Grp{\frac{\snorm{X'}}{1-\delta}}
    \cf{\snorm{X'}\ge (1-\delta)b}
  }.
\end{align*}
If $\mean[g(c\snorm{X'})]<\infty$ for some $c>1$, then by choosing
$\delta>0$ with $1/(1-\delta)<c$, it is seen \eqref{e:limit-2} holds.
Moreover, since $g(\snorm {x-y}) \le g(2\snorm{x}) + g(2\snorm{y})$,
$\mean[g(c\snorm {X'})] \le 2 \mean[g(2c \snorm X)]$, and hence
\eqref{e:limit-2} holds once $\mean[g(2c\snorm X)] < 
\infty$.

Finally, if $X$ is symmetric, then $U_n$ is symmetric and from
\eqref{e:joint} and \cite{ledoux:91}, Proposition 2.3,  
$c_0 n\pr\{\snorm{U_n}\ge x\} \le 2 \pr\{\snorm X\ge x\}$.  Then
by similar argument, \eqref{e:limit-2} holds once $\mean[g(\snorm
X)]<\infty$.

\subsection{Lemmas for univariate case}
\begin{proof}[Proof of Lemma \ref{lemma:rdf}]
  1) From the assumption, $\int |u|^j\, \ld(\dd u)<\infty$ for all
  $j\ge 2$.  Then by dominated convergence, $\che_\xi\in 
  C^\infty(\Reals)$ with $\che_\xi\Sp j(t) = \int (\cf{j=1}-e^{\iunit
    t u}) (\iunit u)^j \,\nu(\dd u)$ for $j\ge 1$.  By $|1-e^{\iunit
    x}|\le |x|$ for $x\in\Reals$, $|\che_\xi'(t)| \le \kappa_{2,\xi}
  |t|$.  Clearly, $|\che_\xi\Sp j(t)| \le |\kappa|_{j,\xi}$ for $j\ge
  2$.  Since $\chf_{\xi+Z}(t) = \exp(-\che_\xi(t) - \rx^2 t^2/2)$,
  then for $j\ge 0$, $\chf\Sp j_{\xi+Z}(t) = P_j(\che_\xi'(t), \ldots,
  \che_\xi\Sp j(t), t)\, \chf_\xi(t) \exp(-\rx^2 t^2/2)$, where $P_j(z)$
  is a multivariate polynomial in $z=(\eno z {j+1})$ of order $j$.  It
  follows that $|\chf\Sp j_{\xi + Z}(t)| = O(|t|^j e^{-\rx^2 t^2/2})$
  and hence for any $p\ge 0$, $|t|^p  |\chf\Sp j_{\xi+Z}(t)|\to 0$ as
  $|t|\toi$, which yields the proof.

  2) For $|t|\ge 1/r$,
  \begin{align*}
    \Re[\che_{X_r}(t)]
    &=
    \int_{|u|<r} (1-\cos t u) \,\ld(\dd u)
    \ge
    \int_{|u|<1/|t|} (1-\cos t u) \,\ld(\dd u).
  \end{align*}
  Then by $1-\cos x\ge C_1 x^2/2$ for $|x|\le 1$,
  \begin{align*}
    \Re[\che_{X_r}(t)]\ge
    \frac{C_1^2 t^2}{2} \int_{u<1/|t|} u^2 \,\ld(\dd u)
    \ge
    L(t,r).
  \end{align*}
  On the other hand, by Cauchy-Schwartz inequality,
  \begin{align*}
    \int_{|t|\ge 1/r} |t|^q |\chf_{X_r}(t)|\,\dd t
    &\le 
    \Grp{
      \int \frac{\dd t}{1+t^2}
    }^{1/2}    
    \Grp{
      \int_{|t|\ge 1/r} (1+t^2) t^{2q} |\chf_{X_r}(t)|^2\,\dd t
    }^{1/2}
    \\
    &\le
    \sqrt{\pi}
    \Grp{
      \int (1+t^2) t^{2q} e^{-2L(t,r)}\,\dd t
    }^{1/2}
  \end{align*}
  Then by \eqref{e:exp-L}, $|t|^q |\chf_{X_r}(t)| \in L^1(\Reals)$
  and the proof follows from Proposition 28.1 of \cite{sato:99}.
\end{proof}

To prove Lemma \ref{lemma:L1}, we need a type of inequalities that
are known (cf.\ \cite{bhattacharya:76}, Lemma 11.6).  Since the
expression of $(\ft f)\Sp j$ becomes complicated rapidly as $j$
increases, the following specific form is used to reduce the maximum
order of derivative involved.
\begin{lemma} \label{lemma:L1-deriv}
  Let $f\in \rdf(\Reals)$ and $\chf(t) = \ft f$.  Then for $j\ge 1$,
  \begin{align*}
    &
    \int |f\Sp j(x)|\,\dd x\\
    &\le \nth{\sqrt{2}}\Sbr{
      \Grp{\int | t^j \chf(t)|^2\,\dd t}^{1/2}
      + j\Grp{\int |t^{j-1}\chf(t)|^2\,\dd t}^{1/2}
      + \Grp{\int |t^j\chf'(t)|^2\,\dd t}^{1/2}
    }. 
  \end{align*}
\end{lemma}
\begin{proof}
  By Cauchy-Schwartz and Minkowski inequalities,
  \begin{align*}
    \int |f\Sp j(x)|\,\dd x
    &\le \Grp{\int \frac{\dd x}{1+x^2}}^{1/2}\Grp{
      \int |f\Sp j(x)|^2 (1+x^2)\,\dd x
    }^{1/2} \\
    &\le \sqrt{\pi}\Sbr{
      \Grp{\int |f\Sp j(x)|^2\,\dd x}^{1/2} 
      +
      \Grp{\int |x f\Sp j(x)|^2\,\dd x}^{1/2}
    }\\
    &
    = \nth{\sqrt{2}}\Sbr{
      \Grp{\int | t^j \chf(t)|^2\,\dd t}^{1/2}
      + \Grp{\int |(t^j\chf(t))'|^2\,\dd t}^{1/2}
    },
  \end{align*}
  where the last line follows from Plancherel theorem and the fact
  that Fourier transforms of $f\Sp j(x)$ and $x^j f(x)$ are $(-it)^j
  \chf(t)$ and $(-i)^j \chf\Sp j(t)$, respectively
  (\cite{grafakos:08}, p.~100-102).  The proof is complete by applying
  Minkowski inequality to the last integral.
\end{proof}

\begin{proof}[Proof of Lemma \ref{lemma:L1}]
  We only consider the case where $\sppt(\ld)\subset \Reals_+$.  The
  proof for the symmetric case is similar.  For brevity, write
  $f=f_\xi$, $\chf = \chf_\xi$, and $\che= \che_\xi$.  By Lemma
  \ref{lemma:rdf}, $f$, $\chf\in\rdf(\Reals)$.  Write $M = \rx^2 + B
  \sigma(r)^2$.  Then
  \begin{align*}
    \Re[\che(t)] = \Re[\che_W(t/\nu)]
    = 
    \int (1-\cos tu/\nu)[A\ld_r(\dd u) + B\cpg_r(\dd u)]
    + \frac{M t^2}{2 \nu^2}.
  \end{align*}
  If $|t|\le \nu/r$, then $|t u|/\nu\le 1$ for $0\le u< r$.  Since
  $1-\cos x \ge C_1^2 x^2/2$ for $|x|\le 1$,
  \begin{align*}
    \Re[\che(t)]
    &\ge 
    \frac{C_1^2 t^2}{2\nu^2} \int_0^r u^2[A\ld_r(\dd u) + B\cpg_r(\dd u)]
    + \frac{M t^2}{2\nu^2}
    \\
    &= 
    \frac{A C_1^2 \kappa_{2,X_r} t^2 }{2 \nu^2}
    + \frac{B C_1^2 m(r) s(r)^{p+3} t^2}{2 \nu^2}
    \int_0^{r/s(r)} u^{p+2} e^{-u}\,\dd u 
    + \frac{M t^2}{2\nu^2}.
  \end{align*}
  Since $s(r)<r/(p+3)$, 
  \begin{align*}
    \int_0^{r/s(r)} u^{p+2} e^{-u}\,\dd u
    \ge \int_0^{p+3} u^{p+2} e^{-u}\,\dd u
    \ge C_2\Gamma(p+3).
  \end{align*}
  Then by $\Gamma(p+3) m(r) s(r)^{p+3} = \kappa_{2,Y_r}$,
  \begin{align*}
    \Re[\che(t)]
    &\ge 
    \frac{A C_1^2 \kappa_{2, X_r} t^2}{2 \nu^2} +
    \frac{B C_1^2 m(r) s(r)^{p+3} C_2 \Gamma(p+3) t^2} {2 \nu^2} +
    \frac{M t^2}{2\nu^2} \\
    &\ge
    \frac{A C_1^2 \kappa_{2, X_r} t^2}{2 \nu^2} +
    \frac{B C_1^2 C_2 \kappa_{2, Y_r} t^2} {2 \nu^2} +
    \frac{B\sigma(r)^2 t^2}{2\nu^2}
    =
    \frac{D(r)^2 t^2}{2\nu^2}.
  \end{align*}
  If $|t|>\nu/r$, then $r>\nu/|t|$ and
  \begin{align*}
    \Re[\che(t)]
    \ge
    \frac{A C_1^2 t^2}{2\nu^2}\int_{u<\nu/|t|} u^2\ld(\dd u) 
    + \frac{B\sigma(r)^2 t^2}{2\nu^2}
    =
    H(t/\nu,r).
  \end{align*}
  Therefore, for $j\ge 0$,
  \begin{align}
    \int |t^j\chf(t)|^2\,\dd t
    &= 2\int_0^\infty t^{2j} e^{-2\Re[\che(t)]}\,\dd t 
    \nonumber \\
    &\le 2 \int_0^{\nu/r} t^{2j} e^{-D(r)^2 t^2/\nu^2}\,\dd t 
    + 2 \int_{\nu/r}^\infty t^{2j} e^{-2 H(t/\nu,r)}\,\dd t 
    \nonumber \\
    &\le 2 \int_0^\infty t^{2j} e^{-D(r)^2 t^2/\nu^2}\,\dd t 
    + 2 \nu^{2j+1} \int_{1/r}^\infty t^{2j} e^{-2 H(t,r)}\,\dd t
    \nonumber \\
    &\le \frac{\nu^{2j+1} \Gamma(j+1/2)}{D(r)^{2j+1}}
    + 2 \nu^{2j+1} \int_{1/r}^\infty t^{2j} e^{-2 H(t,r)}\,\dd t
    = 2 I_j(r)^2.  \label{e:Ij}
  \end{align}
  Next, $\chf'(t) = -\che'(t) \chf(t)$, with
  \begin{align*}
    \che'(t)
    &=
    \frac{\iunit}{\nu} \int (1-e^{\iunit t u/\nu})
    u\, [A\ld_r(\dd u) + B\cpg_r(\dd u)]
    + \frac{M t}{\nu^2}.
  \end{align*}
  By $|1-e^{\iunit x}|\le |x|$ for all $x\in\Reals$,
  \begin{align*}
    |\che'(t)|
    &\le 
    \frac{t}{\nu^2} \int u^2\, [A\ld_r(\dd u) + B\cpg_r(\dd u)]
    + \frac{M t}{\nu^2}
    \\
    & 
    = \frac{A \kappa_{2,X_r} t}{\nu^2}
    + \frac{B\kappa_{2,Y_r} t}{\nu^2} + 
    \frac{(\rx^2 + B\sigma(r)^2)t }{\nu^2}
    \\
    &
    =
    \frac{(A\kappa_{2,X_r} + B\kappa_{2,T_r})t}{\nu^2}
    +\frac{\rx^2 t}{\nu^2}
    = (1+\rx^2/\nu^2) t.
  \end{align*}
  As a result,
  \begin{align}
   \int |t^j \chf'(t)|^2 \,\dd t
    &
    = \int |t^j \che'(t) \chf(t)|^2\,\dd t
    \nonumber \\
    &\le 
    (1+\rx^2/\nu^2)^2 \int |t^{j+1} \chf(t)|^2 \,\dd t
    \le
    2(1+\rx^2/\nu^2)^2 I_{j+1}(r)^2.
    \label{e:Ij-2}
  \end{align}
  The proof is complete by combining Lemma \ref{lemma:L1-deriv},
  \eqref{e:Ij} and \eqref{e:Ij-2}.
\end{proof}

\subsection{Lemmas for multivariate case}
In this section, we prove Lemmas \ref{lemma:mv-moments} --
\ref{lemma:L1-mv-2}.  Recall that $\std X = \SRM_\tau^{-1} X_\tau$ and
$\std Y = \SRM_\tau^{-1} Y_\tau$ have $\levy$ measures $\std\ld(\dd u,
\dd \theta) = \ld_\tau(\dd v, \dd \omega)$ and $\std\cpg(\dd
u,\dd\theta) = \cpg_\tau(\dd v, \dd\omega)$, respectively, where
$(v,\omega) = J_{\SRM_\tau}^{-1}(u, \theta)$, i.e.
$v = \snorm{\SRM_\tau\theta} u$, $\omega = \SRM_\tau \theta /
\snorm{\SRM_\tau \theta}$.  Note that
\begin{align*}
  v\omega = u \SRM_\tau\theta, 
  \quad
  \snorm{\SRM_\tau^{-1} \omega} = \nth{\snorm{\SRM_\tau\theta}}.
\end{align*}
Then
\begin{align*}
  \std\ld(\dd u, \dd\theta)
  &=
  \cf{v <
    r_\tau(\omega)}\ld(\dd v\gv\omega)\, \nu(\dd\omega)
  =
  \cf{u< \std r(\theta)} \std\ld(\dd u\gv\theta)\,
  \std\nu(\dd \theta),
\end{align*}
where
\begin{align*}
  \std r(\theta) = r_\tau(\omega)/\snorm{\SRM_\tau\theta}
  = \snorm{\SRM_\tau^{-1}\omega} r_\tau(\omega),
  \quad
  \std\ld(\dd u\gv\theta) = \ld(\dd v\gv\omega), \quad
  \std\nu(\dd\theta) = \nu(\dd\omega).
\end{align*}
Similarly,
\begin{align*}
  \std\cpg(\dd u, \dd\theta)
  &=
  m_\tau(\omega) \cf{v>0} 
  v^{p_\tau(\omega)} e^{-v/s_\tau(\omega)}
  \,\dd v\, \nu(\dd \omega)
  \nonumber \\
  &=
  m_\tau(\omega) \cf{u>0} \snorm{\SRM_\tau\theta}^{p_\tau(\omega)+1} 
  u^{p_\tau(\omega)}
  e^{-\snorm{\SRM_\tau\theta} u/s_\tau(\omega)}\, \dd u\, \nu(\dd\omega)
  \nonumber \\
  &=
  \std m(\theta) \cf{u>0} u^{\std p(\theta)}
  e^{-u/\std s(\theta)}\,\dd u\, \std\nu(\dd\theta),
\end{align*}
where 
\begin{align*}
  \std m(\theta) = m_\tau(\omega)
  \snorm{\SRM_\tau\theta}^{p_\tau(\omega)+1}, \quad
  \std p(\theta) = p_\tau(\omega), \quad
  \std s(\theta) = s_\tau(\omega)/\snorm{\SRM_\tau\theta}
  = \snorm{\SRM_\tau^{-1}\omega} s_\tau(\omega).
\end{align*}
Therefore,
\begin{align*}
  \std\cpg(\dd u, \dd\theta) =  \std\cpg(\dd u\gv\theta)
  \std\nu(\theta), \
  \text{with}\
  \std\cpg(\dd u\gv\theta)
  =
  \std m(\theta) \cf{u> 0} u^{\std p(\theta)}
  e^{-u/\std s(\theta)}\,\dd u.
\end{align*}
\begin{lemma} \label{lemma:essential}
  If \eqref{e:mv-p-s-r} holds under $\nu$, then
  \begin{align*}
    \ess\inf \std p(\theta) >-2,
    \quad
    \ess\sup \frac{[\std p(\theta)+3] \std s(\theta)}{\std
      r(\theta)} \le 1, \quad
    \text{under $\std\nu$}.
  \end{align*}
  Furthermore, for $\std R$ and $\std S$ in \eqref{e:R-S-std}
  defined under $\nu$, we have
  \begin{align*}
    \std R = \ess\sup \std r(\theta), \quad
    \std S = \ess\sup [\std p(\theta) + h(d)+1]\std
    s(\theta), \quad
    \text{under $\std\nu$}.
  \end{align*}
  Finally, for $\varrho_\tau(a)$ is defined in \eqref{e:root-std}, we
  have
  \begin{align*}
    \begin{split}
      \varrho_\tau(z) 
      &=
      \text{smallest eigenvalue of } C_1^2 \std M(z)
      \text{ and } \std K
      \\
      \text{with}\ \;
      \std M(z)
      &= 
      \int u^2\theta\theta' \cf{u<\std r(\theta)\wedge
        z}  \std\ld(\dd u, \dd \theta).
    \end{split}
  \end{align*}
\end{lemma}
\begin{proof}
  The lemma is straightforward except for the assertion on
  $\varrho_\tau(a)$.  By change of variable $(u,\theta) =
  J_{\SRM_\tau}(v, \omega)$, 
  \begin{align*}
    \SRM_\tau \std M(z) \SRM_\tau
    &=
    \int (u\SRM_\tau \theta)(u\SRM_\tau \theta)'
    \cf{u< \std r(\theta)\wedge z} \,\std\ld(\dd u,\dd\theta)
    \\
    &=
    \int (v\omega) (v\omega)'
    \cf{
      \snorm{\SRM_\tau^{-1}\omega} v <
      (\snorm{\SRM_\tau^{-1}\omega} r_\tau(\omega)) \wedge z
    } \, \ld_\tau(\dd v, \dd\omega).
  \end{align*}
  Since the right hand side is $M_\tau(z)$, $\std M(z) =
  \SRM_\tau^{-1} M(z) \SRM_\tau^{-1}$.  Then the assertion on
  $\varrho_\tau(z)$ follows.
\end{proof}

\begin{proof}[Proof of Lemma \ref{lemma:mv-moments}]
  Fix $a\ge 3$.  Since the $\levy$ measure $\std\ld$ of $\std X$ has
  bounded support according to condition \eqref{e:R-S-std} and Lemma
  \ref{lemma:essential}, $\mean \snorm{\std X}^a<\infty$.
  \begin{align*}
    \int u^a \std\cpg(\dd u, \dd\theta)
    &=
    \int\std m(\theta) \nu(\dd\theta) \int u^{\std p(\theta)+a}
    e^{-u/\std s(\theta)}\,\dd u \\
    &=
    \int \std m(\theta)
    \Gamma(\std p(\theta)+a+1)\std s(\theta)^{\std p(\theta)+a+1}
    \,\nu(\dd\theta).
  \end{align*}
  Letting
  \begin{align*}
    b = \ess\sup 
    \frac{\Gamma(\std p(\theta)+a+1) \std
      s(\theta)^{a-2}}{\Gamma(\std p(\theta)+3)},
  \end{align*}
  then, by the construction of $\std Y$,
  \begin{align*}
    \int u^a \std\cpg(\dd u, \dd\theta)
    &\le
    b \int \std m(\theta) \Gamma(\std p(\theta)+3)
    \std s(\theta)^{\std p(\theta)+3}\,\nu(\dd\theta)
    \\
    &=
    b \int \nu(\dd\theta) \int u^2 \std\ld_\tau(\dd u\gv\theta)
    = b \int u^2 \std\ld(\dd u,\dd\theta).
  \end{align*}
  Since $b\le \ess\sup[(\std p(\theta)+a) \std s(\theta)]^{a-2}$ and
  by conditions \eqref{e:mv-p-s-r} and \eqref{e:R-S-std} as well as
  Lemma \ref{lemma:essential},
  \begin{align*}
    \ess\sup[(\std p(\theta) + a) \std s(\theta)]
    \le
    \ess\sup \frac{\std p(\theta) + a}{\std p(\theta)+3}
    < \infty,
  \end{align*}
  then $b<\infty$.   Thus $\int u^a
  \std\cpg(\dd u,\dd\theta)<\infty$, giving $\mean\snorm{\std
    Y}^a<\infty$.
\end{proof}

\begin{lemma} \label{lemma:Lambda-derivatives}
  Given $\tau>0$, under conditions \eqref{e:mv-p-s-r} and
  \eqref{e:R-S-std}, $\che_{\std X}$ and $\che_{\std Y}\in
  C^\infty(\Reals^d)$ such that
  \begin{align*}
    |\Pd_j \che_{\std X}(t)| \le d\snorm t, \quad
    |\Pd_j \che_{\std T}(t)| \le d\snorm t, \quad j=1,\ldots, d
  \end{align*}
  and for $|\alpha|\ge 2$,
  \begin{align*}
    |\Pd^\alpha \che_{\std X}(t)|
    \le
    (\std R)^{|\alpha|-2} d, \quad
    |\Pd^\alpha \che_{\std T}(t)|
    \le
    \Grp{
      \std S\ess\sup \frac{\std p(\theta)+|\alpha|}{\std p(\theta)+h(d)+1}
    }^{|\alpha|-2} d.
  \end{align*}
  In particular, if $2\le|\alpha|\le h(d)$, then $|\Pd^\alpha
  \che_{\std T}(t)|\le (\std S)^{|\alpha|-2} d$.
\end{lemma}
\begin{proof}
  By Lemma \ref{lemma:mv-moments}, $\int \snorm u^a \std\lambda(\dd u,
  \dd\theta)<\infty$ for $a\ge 2$, so by dominated convergence,
  $\che_{\std X}\in C^\infty(\Reals^d)$, $\Pd_j\che_{\std X}(t) =
  \int \iunit u\theta_j (1-e^{\iunit u\ip t\theta}) \std\ld(\dd
  u,\dd\theta)$, and $\Pd^\alpha\che_{\std X}(t)= -\int
  (\iunit u \theta)^\alpha e^{\iunit u \ip t\theta} \,\std\ld(\dd
  u,\dd\theta)$ for $|\alpha|\ge 2$.  For $\che_{\std Y}$, similar
  formulas hold.  From $|\Pd_j \che_{\std X}(t)|\le \snorm t\int
  u^2 \,\std\ld(\dd u,\dd\theta)  = \snorm t \tr(\var(\std X))$, the
  first inequality follows.  Likewise, $|\Pd_j \che_{\std
    Y}(t)|\le \snorm t \tr(\var(\std Y))$.  Together with $\nth
  2|\Pd_j (t'\std K t)| \le \snorm{\std K t} \le \snorm t\tr(\std K)$,
  this implies the second inequality.  Next, for $|\alpha|\ge 2$, from
  \begin{align*}
    |\Pd^\alpha \che_{\std X}(t)|
    &\le
    \int u^{|\alpha|}\,
    \cf{u\le \std r(\theta)} \std\ld(\dd u, \dd\theta)
    \\
    &\le
    (\std R)^{|\alpha|-2} \int u^2 \,\std\ld(\dd u,\dd\theta)
    =
    (\std R)^{|\alpha|-2} \tr(\var(\std X)),
  \end{align*}
  the third inequality follows.  Finally, for $|\alpha|\ge 2$, let
  \begin{align*}
    C = 
    \ess\sup \Sbr{\std s(\theta)^{|\alpha|-2} 
      \frac{\Gamma(\std p(\theta)+|\alpha|+1)}{
        \Gamma(\std p(\theta)+3)
      }
    }.
  \end{align*}
  Then by Lemma \ref{lemma:essential},
  \begin{align*}
    |\Pd^\alpha \che_{\std Y}(t)|
    &\le
    \int \std\nu(\dd\theta)
    \int_0^\infty u^{|\alpha|}  \,\std\cpg(\dd u\gv\theta)
    \\
    &=
    \int \std m(\theta)
    [\std s(\theta)]^{\std p(\theta)+|\alpha|+1}
    \Gamma(\std p(\theta)+|\alpha|+1)\,\std\nu(\dd\theta) \\
    &\le
    C
    \int \std m(\theta)
    [\std s(\theta)]^{\std p(\theta) +3}
    \Gamma(\std p(\theta)+3)\,\std\nu(\dd\theta) \\
    &
    = C
    \int u^2  \,\std\cpg(\dd u, \dd\theta)
    =
    C\tr(\var(\std Y)).
  \end{align*}
  Since $|\alpha|\ge 2$, then by \eqref{e:R-S-std} and Lemma
  \ref{lemma:essential},
  \begin{align*}
    C \le \Sbr{\ess\sup 
      \std s(\theta) (\std p(\theta)+|\alpha|)
    }^{|\alpha|-2}
    \le 
    \Grp{\std S\ess\sup 
      \frac{\std p(\theta)+|\alpha|}{\std
        p(\theta)+h(d)+1}
    }^{|\alpha|-2}.
  \end{align*}
  Also, $|\Pd^\alpha \che_{\std Z}(t)| \le \tr(\std K)
  \cf{|\alpha|=2}$.  We therefore get the last inequality.
\end{proof}

\begin{lemma} \label{lemma:Lambda-bounds}
  Let $C_1$ and $C_2$ be the same constants defined in Section
  \ref{ss:error-1d}.  Given $\tau>0$, under conditions
  \eqref{e:mv-p-s-r} and \eqref{e:R-S-std}, the following statements
  are true.
  
  1) If $\snorm t\le 1/\,\std R$, then
  \begin{align*}
    \Re[\che_{\std X}(t)]
    \ge
    \frac{C_1^2 \snorm t^2}{2}, 
    \quad
    \Re[\che_{\std T}(t)]
    \ge
    \frac{C_1^2 C_2 \snorm t^2}{2}.
  \end{align*}
  
  2) If $\snorm t>1/\,\std R$, then
  \begin{align*}
    \Re[\che_{\std X}(t)]
    \ge
    \frac{\varrho_\tau(1/\snorm t)\snorm t^2}{2},
    \quad
    \Re[\che_{\std T}(t)]
    \ge \frac{t'\,\std K\,t}{2}.
  \end{align*}
\end{lemma}
\begin{proof}
  1) Given $\snorm t\le 1/\std R$, by Lemma \ref{lemma:essential},
  for $\std\nu$-a.e.\ $\theta\in S$ and $0<u<\std r(\theta)$,
  $|\ip t\theta u|\le |\ip t\theta|  \std r(\theta)\le 1$, yielding
  $1-\cos(\ip t\theta u) \ge C_1^2 \ip t\theta^2 u^2/2$.  Therefore,
  \begin{align*}
    \Re[\che_{\std X}(t)]
    &=
    \int [1-\cos(\ip t\theta u)] \cf{u<\std r(\theta)}
    \std\ld(\dd u\gv\theta) \,\std\nu(\dd\theta)
    \\
    &\ge
    \frac{C_1^2}{2} \int \ip t\theta^2 u^2 
    \cf{u<\std r(\theta)}
    \std\ld(\dd u\gv\theta)\, \std\nu(\dd\theta)
    =
    \frac{C_1^2}{2} \int \ip t\theta^2 u^2
    \,\std\ld(\dd u, \dd \theta).
  \end{align*}
  Since the last integral equals $t'\var(\std X) t = \snorm t^2$,
  we get the first inequality in 1).  Next, by \eqref{e:R-S-std} and
  Lemma \ref{lemma:essential}, for $\std\nu$-a.e.\ $\theta\in S$ and
  $0<u\le \std s(\theta) [\std p(\theta)+3]$, $|\ip t\theta u| \le
  \snorm t u \le \std r(\theta)/\std R\le 1$.  Then
  \begin{align*}
    \int_0^\infty [1-\cos(\ip t\theta u)] \std\cpg(\dd u\gv\theta)
    &\ge
    \int_0^{\std s(\theta)[\std p(\theta)+3]}
    [1-\cos(\ip t\theta u)]  \,\std\cpg(\dd u\gv\theta)
    \\
    &\ge
    \frac{C_1^2}{2}
    \int_0^{\std s(\theta)[\std p(\theta)+3]}
    \ip t\theta^2 u^2 \, \std\cpg(\dd u\gv\theta)
    \\
    &\ge
    \frac{C_1^2 C_2}{2}
    \int_0^\infty
    \ip t\theta^2 u^2 \,\std\cpg(\dd u\gv\theta).
  \end{align*}
  Then
  \begin{align*}
    \Re[\che_{\std Y}(t)]
    &=
    \int [1-\cos(\ip t\theta u)] \,\std \cpg(\dd u\gv\theta)
    \, \std\nu(\dd\theta)
    \\
    &\ge
    \frac{C_1^2 C_2}{2}
    \int \ip t\theta^2 u^2 \,\std\cpg(\dd u,\dd\theta)
    =
    \frac{C_1^2 C_2}{2} t'\var(\std Y) t
  \end{align*}
  and hence
  \begin{align*}
    \Re[\che_{\std T}(t)]
    &=
    \Re[\che_{\std Y}(t)] + \Re[\che_{\std Z}(t)]
    \\
    &\ge
    \frac{C_1^2 C_2}{2} t'\var(\std Y) t + \frac{1}{2} t'\std K t
    \ge
    \frac{C_1^2 C_2}{2} t'[\var(\std Y) + \std K] t.
  \end{align*}
  Since $\var(\std Y) + \std K = \var(\std T) = I$, then the second
  inequality in 1) follows.

  2) The first inequality follows from 
  \begin{align*}
    \Re[\che_{\std X}(t)]
    &\ge
    \int [1-\cos(\ip t\theta u)]
    \cf{u<\std r(\theta)\wedge \nth{\snorm{t}}}
    \std\ld(\dd u, \dd\theta) 
    \\
    &\ge
    \frac{C_1^2}{2} \int \ip t\theta^2 u^2 
    \cf{u<\std r(\theta)\wedge \nth{\snorm{t}}}
    \std\ld(\dd u, \dd\theta)
    =
    \frac{t'\std M(1/\snorm t) t}{2}
  \end{align*}
  and Lemma \ref{lemma:essential}, while the second one from
  $\Re[\che_{\std T}(t)] \ge \Re[\che_{\std Z}(t)] = t'\std K
  t/2$.
\end{proof}

\begin{proof}[Proof of Lemma \ref{lemma:mv-smooth}]
  The proof for 1) is completely similar to that for 1) of Lemma
  \ref{lemma:rdf}, except that it is based on Lemma
  \ref{lemma:Lambda-derivatives}.  To prove 2), let $k=h(d)$.  By
  Cauchy-Schwartz inequality,
  \begin{align*}
    \int \snorm t^q |\chf_{\std X}(t)|\,\dd t
    \le
    \Grp{
      \int \snorm t^{2q}(1+\snorm t^{2k}) |\chf_{\std X}(t)|^2\,\dd t
    }^{1/2}
    \Grp{
      \int \frac{\dd t}{1+\snorm t^{2k}}
    }^{1/2}.
  \end{align*}
  The second factor on the right hand side is finite.  By 2) of
  Lemma \ref{lemma:Lambda-bounds}, for $\snorm t > 1/\std R$,
  $|\chf_{\std X}(t)| \le \exp\{-\Re[\che_{\std X}(t)]\}
  \le \exp\{-\varrho_\tau(1/\snorm t) \snorm t^2/2\}$.  Then
  \begin{align*}
    \int_{\snorm t>1/\std R}
    \snorm t^{2q}(1+\snorm t^{2k}) |\chf_{\std X}(t)|^2\,\dd t
    &\le 
    \int_{\snorm t>1/\std R}
    \snorm t^{2q}(1+\snorm t^{2k})
    e^{-\varrho_\tau(1/\snorm t) \snorm t^2}\,\dd t
    \\
    &\le
    c(d) \int s^{2q+d-1} (1+s^{2k})
    e^{-\varrho_\tau(1/s) s^2}\,\dd s,
  \end{align*}
  where $c(d)$ is a universal constant.  Then by condition
  \eqref{e:exp-L-mv}, $\snorm t^q |\chf_{\std X}(t)|\in L^1(\Reals^d)$
  and the proof follows from Proposition 28.1 of \cite{sato:99}. 
\end{proof}

To prove Lemma \ref{lemma:L1-mv-2}, we use the following variant of
Lemma 11.6 of \cite{bhattacharya:76} which involves lower order of
partial derivatives.
\begin{lemma} \label{lemma:L1-mv}
  For $m\ge 1$, there is a constant $c_1(d, m)$, such that for $f\in
  \rdf(\Reals^d)$,
  \begin{align*}
    \max_{|\alpha|=m} \int |f\Sp\alpha(x)|\,\dd x 
    \le
    c_1(d,m) \max\Cbr{
      \sqrt{\int |t^\beta
        \Pd_i^j \ft f(t)|^2\,\dd t}:
      |\beta|\le m,\, 1\le i\le d,\, 0\le j\le h(d)
    }.
  \end{align*}
\end{lemma}
\begin{proof}
  Denote $k=h(d)$ and $w(x) = x_1^{2k} + \cdots + x_n^{2k}$.  By
  Cauchy-Schwartz inequality,
  \begin{align*}
    \int |f\Sp\alpha(x)|\,\dd x 
    &\le
    \sqrt{\int \frac{\dd x}{1+w(x)}}
    \sqrt{\int |f\Sp\alpha(x)|^2 (1+w(x))\,\dd x}.
  \end{align*}
  First,
  \begin{align*}
    \int \frac{\dd x}{1+w(x)}
    &\le
    d  \int \frac{\dd x_1}{1+x_1^{2k}}
    \int_{0\le |x_i| \le |x_1|} \dd x_2\cdots\dd x_k
    =
    d 2^d \int_0^\infty \frac{x^{d-1}\,\dd x}{1+x^{2k}} = c'(d).
  \end{align*}
  Next, by Plancherel theorem and properties of the Fourier
  transform (\cite{grafakos:08}, p.~100-102),
  \begin{align*}
    \int |f\Sp\alpha(x)|^2 (1+w(x))\,\dd x
    &=
    \int |f\Sp\alpha(x)|^2 \, \dd x
    + \sum_{i=1}^d \int |x_i^k f\Sp\alpha(x)|^2 \, \dd x 
    \\
    &=
    \int |t^\alpha \ft f(t)|^2 \, \dd t
    + \sum_{i=1}^d \int |\Pd_i^k [t^\alpha \ft f(t)]|^2 \,
    \dd t,
  \end{align*}
  so by $\sqrt{a+b}\le \sqrt{a}+\sqrt{b}$ for $a,b\ge 0$,
  \begin{align*}
    \sqrt{
      \int |f\Sp\alpha(x)|^2 (1+w(x))\,\dd x
    }
    \le
    \sqrt{\int |t^\alpha \ft f(t)|^2 \, \dd t}
    + \sum_{i=1}^d \sqrt{
      \int |\Pd_i^k [t^\alpha \ft f(t)]|^2 \, \dd t
    }.
  \end{align*}
  Since for each $i=1,\ldots,d$,
  \begin{align*}
    \Pd_i^k [t^\alpha \ft f(t)]
    &=
    \sum_{0\le j\le \alpha_i\wedge k} \binom{k}{j}
    \Pd^j_i t^\alpha \cdot\Pd^{k-j}_i \ft f(t)
    \\
    &=
    \sum_{0\le j\le \alpha_i\wedge k} \binom{k}{j}
    \frac{\alpha_i!}{(\alpha_i-j)!}
    t^\alpha t_i^{-j} \Pd^{k-j}_i \ft f(t),
  \end{align*}
  by Minkowski inequality, the desired inequality follows.
\end{proof}

Finally, notice that for $k\ge 1$, there is a unique multivariate
polynomial of $x=(\eno x k)\in \Coms^k$,
\begin{align*}
  P_k(x) = \sum_{i\in I_k} a_i x^i, \quad
  \text{with } a_i\in \Ints_+,
\end{align*}
where $I_k = \{i=(\eno i k)\in \Ints_+^k: \sum_{j=1}^k j i_j = k\}$,
such that for any $k$-times differentiable function $\che$ on
$\Reals$, letting $\chf = \exp(\che)$,
\begin{align} \label{e:derivative-cf}
  \chf\Sp k =
  P_k(\che', \che'', \ldots, \che\Sp k) \chf
  = \sum_{i\in I_k} a_{i_1 i_2 \cdots i_k} (\che')^{i_1}
  (\che'')^{i_2} \cdots (\che\Sp k)^{i_k} \chf.
\end{align}
Note that if $\che(t) = - t^2/2$, then $(-1)^k P_k(\che'(t),
\ldots, \che\Sp k(t))$ is the $k\th$-order Hermite polynomial.

\begin{proof}[Proof of Lemma \ref{lemma:L1-mv-2}]
  Denote $k=h(d)$.  By Lemma \ref{lemma:mv-smooth}, $f_\xi\in
  \rdf(\Reals^d)$.  By Lemma \ref{lemma:Lambda-derivatives}, for
  $i=1,\ldots, d$
  \begin{align*}
    |\Pd_i \che_\xi(t)|
    \le A |\Pd_i \che_{\std X}(t)| 
    + B|\Pd_i \che_{\std T}(t)| + \rx^2\snorm t
    \le (d+\rx^2)\snorm t
  \end{align*}
  and for $2\le j\le k$,
  \begin{align*}
    |\Pd_i^j \che_\xi(t)|
    &\le
    A |\Pd_i^j \che_{\std X}(t)| 
    + B|\Pd_i^j \che_{\std T}(t)| + \rx^2\cf{j=2}
    \\
    &\le
    (\std R)^{j-2} A d + (\std S)^{j-2} B d + \rx^2\cf{j=2}
    \\
    &\le
    Ad + B d + \rx^2\cf{j=2}
    \le
    d + \rx^2.
  \end{align*}
  Thus, by \eqref{e:derivative-cf}, $|\Pd_i^j \chf_\xi(t)| \le
  (d+\rx^2)^j P_j(\snorm t, 1, 1, \ldots, 1) |\chf_\xi(t)|$.  Since
  $P_j(x,1,\ldots,1)$ is a $j\th$-order polynomial of $x\in \Coms$
  with coefficients only depending on $j$, there is a polynomial $h(x)
  = h_{d,m}(x)$ of order no greater than $2m+2k$ with coefficients
  only depending on $(d,m)$, such that for all $\beta$ with
  $|\beta|\le m$, $i=1,\ldots, d$ and $0\le j\le k$,
  \begin{align*}
    \int |t^\beta \Pd_i^j\chf_\xi(t)|^2\,\dd t
    \le
    (d+\rx^2)^{2k} \int h(\snorm t) |\chf_\xi(t)|^2\,\dd t.
  \end{align*}
  Write
  \begin{align*}
    I_1
    =
    \int \cf{\snorm t\le 1/\std R} 
    h(\snorm t) |\chf_\xi(t)|^2\,\dd t,
    \quad
    I_2
    =
    \int \cf{\snorm t> 1/\std R}
    h(\snorm t) |\chf_\xi(t)|^2\,\dd t.
  \end{align*}
  By 1) of Lemma \ref{lemma:Lambda-bounds}, for $\snorm t\le 1/\std
  R$,
  \begin{align*}
    \Re[\che_\xi(t)] \ge
    A \Re[\che_{\std X}(t)] + B\Re[\che_{\std T}(t)]
    \ge C_1^2 C_2\snorm t^2/2.
  \end{align*}
  Then
  \begin{align*}
    I_1
    \le
    \int h(\snorm t) \exp\{-2\Re[\che_\xi(t)]\}\,\dd t
    \le
    \int h(\snorm t) \exp\{-C_1^2 C_2 \snorm t^2\}\,\dd t
    = c'(d,m).
  \end{align*}
  On the other hand, by 2) of Lemma \ref{lemma:Lambda-bounds}, for
  $\snorm t>1/\std R$,
  \begin{align*} 
    \Re[\che_\xi(t)]\ge
    A \Re[\che_{\std X}(t)] + B \Re[\che_{\std T}(t)]
    \ge \varrho_\tau(1/\snorm t)\snorm t^2/2.
  \end{align*}
  Therefore, by change of variable $t=s\omega$ with $s>0$ and
  $\omega\in S$,
  \begin{align*}
    I_2
    &\le
    \int \cf{\snorm t>1/\std R} h(\snorm t)
    \exp(-\varrho_\tau(1/\snorm t) \snorm t^2)\,\dd t
    \\
    &=
    c''(d)
    \int_{1/\std R}^\infty s^{d-1}h(s) e^{-\varrho_\tau(1/s) s^2}
    \,\dd s.
  \end{align*}
  As a result, for all $\beta$ with
  $|\beta|\le m$, $i=1,\ldots, d$ and $0\le j\le k$,
  \begin{align*}
    \int |t^\beta \Pd_i^j\chf_\xi(t)|^2\,\dd t
    \le
    (1+\rx^2/d)^{2k} d^{2k} \Grp{
      c'(d,m) + c''(d)\int_{1/\std R}^\infty s^{d-1} h(s)
      e^{-\varrho_\tau(1/s) s^2}\,\dd s
    }.
  \end{align*}
  Combining the bound with Lemma \ref{lemma:L1-mv}, the proof is
  complete.
\end{proof}

\bibliographystyle{acmtrans-ims}
\bibliography{Levy,ldpdb,LimitTheorems,Misc,Estimate}

\end{document}